\documentclass{article}
\usepackage{amscd,amsmath,amssymb,amsthm,enumerate,epsfig,float,graphicx,hyperref,makecell,mathrsfs,rotating,tikz,wasysym}

\newtheorem{theorem}{Theorem}[section]
\newtheorem{proposition}[theorem]{Proposition}

\newtheorem{remark}[theorem]{Remark}
\theoremstyle{definition}
\newtheorem{definition}[theorem]{Definition}

\makeatletter
\newtheorem*{rep@theorem}{\rep@title}
\newcommand{\newreptheorem}[2]{%
\newenvironment{rep#1}[1]{%
 \def\rep@title{#2 \ref{##1}${}'$}%
 \begin{rep@theorem}}%
 {\end{rep@theorem}}}
\makeatother

\newtheorem{lemma}[theorem]{Lemma}
\newreptheorem{lemma}{Lemma}
\numberwithin{equation}{subsection}

\title{On the Reducibility of Scalar Generalized Verma Modules of Abelian Type}
\author{Haian HE}
\date{}

\begin{document}

\maketitle

\begin{abstract}
A parabolic subalgebra $\mathfrak{p}$ of a complex semisimple Lie algebra $\mathfrak{g}$ is called a parabolic subalgebra of abelian type if its nilpotent radical is abelian. In this paper, we provide a complete characterization of the parameters for scalar generalized Verma modules attached to parabolic subalgebras of abelian type such that the modules are reducible. The proofs use Jantzen's simplicity criterion, as well as the Enright-Howe-Wallach classification of unitary highest weight modules.
\end{abstract}

\section{Introduction}

Let $\mathfrak{g}$ be a complex semisimple Lie algebra, and fix a Cartan subalgebra $\mathfrak{h}$. Denote by $\Phi$ (respectively, $\Phi^+$) the set of roots (respectively, positive roots) of $(\mathfrak{g},\mathfrak{h})$. Let $\mathfrak{p}$ be a maximal parabolic subalgebra of $\mathfrak{g}$ with standard Levi decomposition $\mathfrak{p}=\mathfrak{l}+\mathfrak{u}_+$ with respect to $(\mathfrak{h},\Phi^+)$, where $\mathfrak{l}$ is a Levi factor and $\mathfrak{u}_+$ is the nilpotent radical. Let $\Phi_\mathfrak{l}$ be the set of roots of $(\mathfrak{l},\mathfrak{h})$, and put $\Phi_\mathfrak{l}^+=\Phi_\mathfrak{l}\cap\Phi^+$. If $\lambda\in\mathfrak{h}^*$ is $\Phi_\mathfrak{l}^+$-dominant integral, let $F_\lambda$ be the finite-dimensional complex simple $\mathfrak{l}$-module with highest weight $\lambda$. By letting the nilpotent radical $\mathfrak{u}_+$ act by 0, $F_\lambda$ is also a module of the parabolic subalgebra $\mathfrak{p}$. Now the generalized Verma module of $\mathfrak{g}$ attached to $\mathfrak{p}$ with the parameter $\lambda$ is defined to be\[M_\mathfrak{p}^\mathfrak{g}(\lambda):=U(\mathfrak{g})\otimes_{U(\mathfrak{p})}F_\lambda\]where $U(\mathfrak{g})$ (respectively, $U(\mathfrak{p})$) is the universal enveloping algebra of $\mathfrak{g}$ (respectively, $\mathfrak{p}$). When $\dim_\mathbb{C}F_\lambda=1$, $M_\mathfrak{p}^\mathfrak{g}(\lambda)$ is called a scalar generalized Verma module; when $\mathfrak{p}$ is a Borel subalgebra, it is just a Verma module. As is known, generalized Verma modules form a fundamental and distinguished class of objects in the parabolic BGG category $\mathcal{O}^\mathfrak{p}$. A universal property is that each highest weight module in $\mathcal{O}^\mathfrak{p}$ can be covered by a generalized Verma module $M_\mathfrak{p}^\mathfrak{g}(\lambda)$ for some $\lambda$. More details about generalized Verma modules can be found in [\textbf{H}].

The reducibility of generalized Verma modules is an interesting problem, which is much more complicated than the problem for Verma modules. The study of this problem has a long history. In 1977, a simplicity criterion for generalized Verma modules was shown by J. C. Jantzen in [\textbf{J}], and it remains one of the most well-known and widely used results along these lines. After that, T. Enright, R. Howe, and N. Wallach worked out the parameters of reducible generalized Verma modules related to unitary highest weight modules in 1983 [\textbf{EHW}]. Recently, A. Kamita used $b$-functions to describe the reducibility of generalized Verma modules in [\textbf{Ka}]. Apart from directly studying the reducibility problem, many mathematicians studied homomorphisms between generalized Verma modules, which gave some results on reducibility of generalized Verma modules indirectly, e.g., [\textbf{B}], [\textbf{F}], [\textbf{G}], and [\textbf{M}].

A parabolic subalgebra $\mathfrak{p}$ of a complex semisimple Lie algebra $\mathfrak{g}$ is called a parabolic subalgebra of abelian type if its nilpotent radical is abelian. We now explain how to characterize the parabolic subalgebras of abelian type in simple Lie algebras in terms of Hermitian symmetric pairs. Suppose that $G$ is a connected real simple Lie group with center $Z$, and let $K$ be a closed maximal subgroup of $G$ with $K/Z$ compact. Let $\mathfrak{g}$ be the complexified Lie algebra of $G$. A unitary representation $(\pi,V)$ of $G$ such that the underlying $(\mathfrak{g},K)$-module is a simple quotient of a Verma module of $\mathfrak{g}$ is called a unitary highest weight module. Harish-Chandra showed that $G$ admits non-trivial unitary highest weight modules precisely when $(G,K)$ is a Hermitian symmetric pair ([\textbf{HC1}] and [\textbf{HC2}]). Now denote by $\mathfrak{l}$ the complexified Lie algebra of $K$. Fix a Cartan subalgebra $\mathfrak{h}$ of $\mathfrak{l}$; then our assumptions on $G$ imply that $\mathfrak{h}$ is also a Cartan subalgebra of $\mathfrak{g}$. Since $(\mathfrak{g},\mathfrak{l})$ is a Hermitian symmetric pair, we may choose a simple root system $\Delta$ of $(\mathfrak{g},\mathfrak{h})$ such that the standard Borel subalgebra $\mathfrak{b}$ satisfying that $\mathfrak{p}:=\mathfrak{l}+\mathfrak{b}$ is a parabolic subalgebra of $\mathfrak{g}$. According to the classification of the Hermitian symmetric pairs [\textbf{W}], the nilpotent radical of $\mathfrak{p}$ must be abelian and hence $\mathfrak{p}$ is a parabolic subalgebra of abelian type. Conversely, suppose that $\mathfrak{g}$ is a complex simple Lie algebra and $\mathfrak{p}$ is a parabolic subalgebra of abelian type of $\mathfrak{g}$, and then $\mathfrak{p}$ is automatically a maximal parabolic subalgebra. It follows from [\textbf{RRS}] that there exists a real form $\mathfrak{g}_\mathbb{R}$ of $\mathfrak{g}$ such that $G_\mathbb{R}/(G_\mathbb{R}\cap P)$ is a Hermitian symmetric space, where $G_\mathbb{R}$ and $P$ are the subgroups of the adjoint group $G=\mathrm{Int}(\mathfrak{g})$ with Lie algebras $\mathfrak{g}_\mathbb{R}$ and $\mathfrak{p}$. The group $K:=G_\mathbb{R}\cap P$ is a maximal compact subgroup of $G_\mathbb{R}$, and its complexified Lie algebra gives the Levi factor, denoted by $\mathfrak{l}$, of $\mathfrak{p}$.

If $\mathfrak{p}$ is a parabolic subalgebra of abelian type, we call a generalized Verma module $M_\mathfrak{p}^\mathfrak{g}(\lambda)$ attached to $\mathfrak{p}$ a generalized Verma module of abelian type. In [\textbf{EHW}], T. Enright, R. Howe, and N. Wallach worked out all the parameters $\lambda$ such that the generalized Verma module $M_\mathfrak{p}^\mathfrak{g}(\lambda)$ of abelian type is reducible and its unique simple quotient $L(\lambda)$ is unitarizable. However, if $L(\lambda)$ is not unitarizable, it is not known for which $\lambda$ $M_\mathfrak{p}^\mathfrak{g}(\lambda)$ is simple. We shall answer this question when $M_\mathfrak{p}^\mathfrak{g}(\lambda)$ is a scalar generalized Verma module, i.e., $F_\lambda=\mathbb{C}_\lambda$ is a one-dimensional complex $\mathfrak{l}$-module, which is equivalent to saying that $(\lambda,\alpha)=0$ for all $\alpha\in\Phi_\mathfrak{l}$, where $(-,-)$ is the inner product on $\mathfrak{h}^*$ induced by the Killing form of $\mathfrak{g}$. We shall recall the techniques in [\textbf{EHW}] in Section 2.2.

Let $\gamma$ denote the unique maximal root in $\Phi^+$. Denote by $\zeta$ the unique element in $\mathfrak{h}^*$ such that $(\zeta,\alpha)=0$ for all $\alpha\in\Phi_\mathfrak{l}$ and $\frac{2(\zeta,\gamma)}{(\gamma,\gamma)}=1$. It is obvious that the parameters $\lambda$ of scalar generalized Verma modules must satisfy $\lambda=c\zeta$ for some $c\in\mathbb{C}$.

Write $\Phi_\mathfrak{u}:=\Phi\setminus\Phi_\mathfrak{l}$, and denote by $\Delta$ the simple root system of $\Phi^+$. Then each standard maximal parabolic subalgebra with respect to $(\mathfrak{h},\Delta)$ is determined by the unique simple root in $\Delta_\mathfrak{u}:=\Phi_\mathfrak{u}\cap\Delta$. Now we can state our main result.
\begin{theorem}
Let $\mathfrak{g}$ be a complex simple Lie algebra, and let $\mathfrak{p}$ be a parabolic subalgebra of abelian type of $\mathfrak{g}$. Fix a Cartan subalgebra $\mathfrak{h}\subseteq\mathfrak{p}$, and choose a simple root system $\Delta$ for $(\mathfrak{g},\mathfrak{h})$ such that $\mathfrak{p}$ contains the standard Borel subalgebra with respect to $(\mathfrak{h},\Delta)$. Then the parameters $\lambda\in\mathfrak{h}^*$ such that $M_\mathfrak{p}^\mathfrak{g}(\lambda)$ is a reducible scalar generalized Verma module are precisely given case by case, according to the Hermitian symmetric pairs of compact type, as follows.
\begin{enumerate}[(i)]
\item $(SU(p+q),S(U(p)\times U(q)))$ for $p,q\geq1$: $\Delta=\{e_i-e_{i+1}\mid1\leq i\leq p+q-1\}$; $\Delta_\mathfrak{u}=\{e_p-e_{p+1}\}$; $\zeta=\displaystyle{\frac{q}{p+q}\sum_{i=1}^pe_i-\frac{p}{p+q}\sum_{j=p+1}^{p+q}e_j}$; $\lambda=c\zeta$ with $c\in1-\min\{p,q\}+\mathbb{Z}_{\geq0}$.
\item $(Sp(n),U(n))$ for $n\geq2$: $\Delta=\{e_i-e_{i+1}\mid1\leq i\leq n-1\}\cup\{2e_n\}$; $\Delta_\mathfrak{u}=\{2e_n\}$; $\zeta=\displaystyle{\sum_{i=1}^ne_i}$; $\lambda=c\zeta$ with $c\in\frac{1-n}{2}+\frac{1}{2}\mathbb{Z}_{\geq0}$.
\item $(SO(2n+1),SO(2)\times SO(2n-1))$ for $n\geq2$: $\Delta=\{e_i-e_{i+1}\mid1\leq i\leq n-1\}\cup\{e_n\}$; $\Delta_\mathfrak{u}=\{e_1-e_2\}$; $\zeta=e_1$; $\lambda=c\zeta$ with $c\in\mathbb{Z}_{\geq0}\cup(\frac{3}{2}-n+\mathbb{Z}_{\geq0})$.
\item $(SO(2n),SO(2)\times SO(2n-2))$ for $n\geq2$: $\Delta=\{e_i-e_{i+1}\mid1\leq i\leq n-1\}\cup\{e_{n-1}+e_n\}$; $\Delta_\mathfrak{u}=\{e_1-e_2\}$; $\zeta=e_1$; $\lambda=c\zeta$ with $c\in2-n+\mathbb{Z}_{\geq0}$.
\item $(SO(2n),U(n))$ for $n\geq2$: $\Delta=\{e_i-e_{i+1}\mid1\leq i\leq n-1\}\cup\{e_{n-1}+e_n\}$; $\Delta_\mathfrak{u}=\{e_{n-1}+e_n\}$; $\zeta=\displaystyle{\frac{1}{2}\sum_{i=1}^ne_i}$; $\lambda=c\zeta$ with $c\in2[\frac{3-n}{2}]+\mathbb{Z}_{\geq0}$ where $[x]$ denotes the largest integer not greater than $x\in\mathbb{R}$.
\item $(E_{6(-78)},SO(2)\times SO(10))$: $\Delta=\{\alpha_i\mid1\leq i\leq6\}$ with $\alpha_1=\frac{1}{2}(e_1-e_2-e_3-e_4-e_5-e_6-e_7+e_8)$, $\alpha_2=e_1+e_2$, $\alpha_i=e_{i-1}-e_{i-2}$ ($3\leq i\leq6$); $\Delta_\mathfrak{u}=\{\alpha_1\}$; $\zeta=\frac{2}{3}(-e_6-e_7+e_8)$; $\lambda=c\zeta$ with $c\in-3+\mathbb{Z}_{\geq0}$.
\item $(E_{7(-78)},SO(2)\times E_{6(-78)})$: $\Delta=\{\alpha_i\mid1\leq i\leq7\}$ with $\alpha_i$ ($1\leq i\leq6$) same as in (vi) and $\alpha_7=e_6-e_5$; $\Delta_\mathfrak{u}=\{\alpha_7\}$; $\zeta=e_6-\frac{1}{2}e_7+\frac{1}{2}e_8$; $\lambda=c\zeta$ with $c\in-8+\mathbb{Z}_{\geq0}$.
\end{enumerate}
\end{theorem}
\begin{remark}
An observation from Theorem 1.1 is that in each case, the set of parameters for which the corresponding scalar generalized Verma modules of abelian type are reducible constitutes a finite set and a semi-infinite arithmetic progression.
\end{remark}
\begin{remark}
Although Theorem 1.1 only discusses complex simple Lie algebras, people may deduce the results for complex semisimple Lie algebras immediately. In fact, suppose that $\mathfrak{g}=\displaystyle{\bigoplus_{j}\mathfrak{g}_j}$ denotes the decomposition of a finite-dimensional complex semisimple Lie algebra into its simple ideals, and $\mathfrak{p}$ is a parabolic subalgebra of $\mathfrak{g}$. Then
\begin{itemize}
\item $\mathfrak{p}=\displaystyle{\bigoplus_{j}\mathfrak{p}_j}$, with each $\mathfrak{p}_j$ a parabolic subalgebra of $\mathfrak{g}_j$.
\item $\mathfrak{p}$ has abelian (respectively, nilpotent) radical if and only if each $\mathfrak{p}_j$ has abelian (respectively, nilpotent) radical.
\item $\lambda=\displaystyle{\sum_{j}\lambda_j}$ is a dominant integral highest weight with respect to $\mathfrak{p}$ if and only if each $\lambda_j$ is dominant integral highest weight with respect to $\mathfrak{p}_j$.
\item For the above data, $M_\mathfrak{p}^\mathfrak{g}(\lambda)=\displaystyle{\bigotimes_jM_{\mathfrak{p}_j}^{\mathfrak{g}_j}(\lambda_j)}$.
\item $M_\mathfrak{p}^\mathfrak{g}(\lambda)$ is a scalar (simple) generalized Verma module if and only if each $M_{\mathfrak{p}_j}^{\mathfrak{g}_j}(\lambda_j)$ is a scalar (simple) generalized Verma module.
\end{itemize}
From these points, we actually obtain the parameters of all the reducible scalar generalized Verma modules of abelian type for complex semisimple Lie algebras.
\end{remark}
The reducibility problem can be studied for scalar generalized Verma modules attached to arbitrary parabolic algebras. Concretely, let $\mathfrak{p}=\mathfrak{l}+\mathfrak{u}_+$ be a parabolic subalgebra of a semisimple Lie algebra $\mathfrak{g}$, where $\mathfrak{l}$ is a Levi factor and $\mathfrak{u}_+$ is the nilpotent radical. Put $\mathfrak{u}_0=\mathfrak{u}_+$ and $\mathfrak{u}_k=[\mathfrak{u}_+, \mathfrak{u}_{k-1}]$ for positive integer $k$. We call $\mathfrak{u}_k$ the $k$th-step of $\mathfrak{u}_+$ for a nonnegative integer $k$. The nilpotent Lie algebra $\mathfrak{u}_+$ is called $k$-step nilpotent if $\mathfrak{u}_{k-1}\neq0$ and $\mathfrak{u}_k=0$. If the nilpotent radical $\mathfrak{u}_+$ of the parabolic subalgebra $\mathfrak{p}$ is $k$-step nilpotent, then we say that $\mathfrak{p}$ is a parabolic subalgebra of $k$-step nilpotent type. In particular, if $\mathfrak{u}_0=\mathfrak{u}_+=0$, then $\mathfrak{p}=\mathfrak{g}$, and hence generalized Verma modules attached to $\mathfrak{p}$ are just finite-dimensional complex simple modules of $\mathfrak{g}$, whose reducibility problems are well-known [\textbf{Kn}, Theorem 5.5]. If $\mathfrak{p}$ is a parabolic subalgebra of $1$-step nilpotent type, then it is just a parabolic subalgebra of abelian type, the scalar generalized Verma modules attached to which are what this paper handles. As for a parabolic subalgebra $\mathfrak{p}$ of $2$-step nilpotent type, if $\dim_\mathbb{C}\mathfrak{u}_1=1$, then $\mathfrak{p}$ is called a parabolic subalgebra of $2$-step nilpotent Heisenberg type; else, it is called a parabolic subalgebra of $2$-step nilpotent non-Heisenberg type. T. Kubo investigated and solved the reducibility problem for scalar generalized Verma modules associated to exceptional simple Lie algebras and maximal parabolic subalgebras of $2$-step nilpotent non-Heisenberg type in [\textbf{Ku}]. Moreover, according to T. Kubo, generalized Verma modules associated to exceptional simple Lie algebras and maximal parabolic subalgebras of $2$-step nilpotent Heisenberg type were studied by R. Zierau, but the results were not published.

The paper is organized as follows. In Section 2, we recall Janzten's criterion for simplicity of generalized Verma modules attached to maximal parabolic subalgebras, and then recall the techniques in [\textbf{EHW}], both of which offer us powerful tools to deal with our problem. In Section 3, we prove Theorem 1.1 in a case-by-case fashion for all seven cases.

\section{Janzten's Criterion and Special Lines}

\subsection{Janzten's Criterion}

In this section, we recall the irreducibility criterion due to J. C. Jantzen for generalized Verma modules. Because we focus on parabolic subalgebras of abelian type, which are automatically maximal parabolic subalgebras, we only state a specialization of Jantzen's criterion for scalar generalized Verma modules attached to maximal parabolic subalgebras $\mathfrak{p}$.

First of all, we fix some notations for the paper. All Lie algebras and modules considered in this paper are over the complex number field $\mathbb{C}$ unless we make any declaration. Denote by $\mathbb{Z}_{\geq0}$ and $\mathbb{Z}_{>0}$ the set of nonnegative integers and the set of positive integers respectively. Let $\mathfrak{g}$ be a finite-dimensional complex semisimple Lie algebra. Choose a Cartan subalgebra $\mathfrak{h}$ of $\mathfrak{g}$, and a Borel subalgebra $\mathfrak{b}$ containing $\mathfrak{h}$. Denote by $\Phi$, $\Phi^+$, and $\Delta$ the root system, the set of positive roots, and the simple root system with respect to $(\mathfrak{g},\mathfrak{b},\mathfrak{h})$ respectively. Let $\mathfrak{p}$ be a maximal parabolic subalgebra such that $\mathfrak{b}\subseteq\mathfrak{p}\subseteq\mathfrak{g}$, and denote by $\mathfrak{p}=\mathfrak{l}+\mathfrak{u}_+$ the Levi decomposition with respect to $(\mathfrak{h},\Delta)$, where $\mathfrak{l}$ is the Levi factor with $\mathfrak{h}\subseteq\mathfrak{l}$ and $\mathfrak{l}+\mathfrak{b}=\mathfrak{p}$, and $\mathfrak{u}_+$ is the nilpotent radical. Now we may denote by $\Phi_\mathfrak{l}$ the root system for $(\mathfrak{l},\mathfrak{h})$, and set $\Phi_\mathfrak{l}^+:=\Phi_\mathfrak{l}\cap\Phi^+$ and $\Delta_\mathfrak{l}:=\Phi_\mathfrak{l}\cap\Delta$. Also, define $\Phi_\mathfrak{u}^+:=\Phi^+\setminus\Phi_\mathfrak{l}^+$. Let $(-,-)$ be the inner product on $\mathfrak{h}^*$ induced by the Killing form of $\mathfrak{g}$, and for $\mu,\nu\in\mathfrak{h}^*$ define $\langle\mu,\nu\rangle:=\frac{2(\mu,\nu)}{(\nu,\nu)}$. Now the set of $\Phi_\mathfrak{l}^+$-dominant integral weights is defined as $\Lambda_\mathfrak{l}^+:=\{\lambda\in\mathfrak{h}^*\mid\langle\lambda,\alpha\rangle\in\mathbb{Z}_{\geq0}\textrm{ for all }\alpha\in\Phi_\mathfrak{l}^+\}$. Moreover, denote by $\rho$ half the sum of the positive roots in $\Phi^+$. For $\alpha\in\Phi$, define a reflection $s_\alpha: \mathfrak{h}^*\rightarrow\mathfrak{h}^*$ given by $s_\alpha(\lambda)=\lambda-\langle\lambda,\alpha\rangle\alpha$ for $\lambda\in\mathfrak{h}^*$, and then denote by $W$ (respectively, $W_\mathfrak{l}$) the Weyl group of $(\mathfrak{g},\mathfrak{h})$ (respectively, $(\mathfrak{l},\mathfrak{h})$) generated by $s_\alpha$ for $\alpha\in\Delta$ (respectively, for $\alpha\in\Delta_\mathfrak{l}$). Let $U(\mathfrak{g})$ (respectively, $U(\mathfrak{p})$) be the universal enveloping algebra of $\mathfrak{g}$ (respectively, $\mathfrak{p}$).

If $\lambda\in\mathfrak{h}^*$ is $\Phi_\mathfrak{l}^+$-dominant integral, let $F_\lambda$ be the finite-dimensional complex simple $\mathfrak{l}$-module with highest weight $\lambda$. By letting the elements in $\mathfrak{u}_+$ act by 0, $F_\lambda$ is induced to a $\mathfrak{p}$-module. Now the generalized Verma module of $\mathfrak{g}$ attached to $\mathfrak{p}$ with the parameter $\lambda$ is defined to be\[M_\mathfrak{p}^\mathfrak{g}(\lambda):=U(\mathfrak{g})\otimes_{U(\mathfrak{p})}F_\lambda.\]When $\dim_\mathbb{C}F_\lambda=1$, i.e., $(\lambda,\alpha)=0$ for all $\alpha\in\Delta_\mathfrak{l}$, it is called a scalar generalized Verma module.
\begin{theorem}[{\cite[Theorem 9.12]{H}}]
Let $\lambda\in\Lambda_\mathfrak{l}^+$. If $\langle\lambda+\rho,\beta\rangle\notin\mathbb{Z}_{>0}$ for all $\beta\in\Phi_\mathfrak{u}^+$, then $M_\mathfrak{p}^\mathfrak{g}(\lambda)$ is simple. The converse also holds if $\lambda+\rho$ is regular.
\end{theorem}
The following notation is important in the statement of Jantzen's criterion. For $\lambda\in\mathfrak{h}^*$, define
\begin{equation}
Y(\lambda):=D^{-1}\displaystyle{\sum_{\omega\in W_\mathfrak{l}}(-1)^{l(\omega)}e^{\omega\lambda}}
\end{equation}
where $l(\omega)$ denotes the length of $\omega\in W_\mathfrak{l}$, $e^\mu$ is a function on $\mathfrak{h}^*$ which takes values 1 at $\mu$ and 0 elsewhere, and $D=e^\rho\displaystyle{\prod_{\alpha\in\Phi^+}(1-e^{-\alpha})}$ is the Weyl denominator. It is clear that $Y(\lambda)$ is the character formula of $M_\mathfrak{p}^\mathfrak{g}(\lambda-\rho)$ if $\lambda$ is $\Phi_\mathfrak{l}^+$-dominant integral.
\begin{proposition}[{\cite[Corollary A.1.5]{Ku}, \cite[Corollary 2.2.10]{M}}]
We have the following two properties:
\begin{enumerate}[(1)]
\item If $\lambda\in\mathfrak{h}^*$ satisfies $(\lambda,\alpha)=0$ for some $\alpha\in\Phi_\mathfrak{l}$, i.e., $\lambda$ is $\Phi_\mathfrak{l}$-singular, then $Y(\lambda)=0$. Conversely, if $\lambda\in\mathfrak{h}^*$ satisfies $\langle\lambda,\alpha\rangle\in\mathbb{Z}\setminus\{0\}$ for all $\alpha\in\Phi_\mathfrak{l}$, i.e., $\lambda$ is $\Phi_\mathfrak{l}$-regular integral, then $Y(\lambda)\neq0$.
\item For $\lambda\in\mathfrak{h}^*$ and $\omega\in W_\mathfrak{l}$, we have $Y(\omega\lambda)=(-1)^{l(\omega)}Y(\lambda)$.
\end{enumerate}
\end{proposition}
Denote by $s_\beta$ the reflection corresponding to $\beta\in\Phi$ in $W$. Set
\begin{equation}
S_\lambda:=\{\beta\in\Phi_\mathfrak{u}^+\mid\langle\lambda+\rho,\beta\rangle\in\mathbb{Z}_{>0}\}.
\end{equation}
Then Jantzen's criterion for the scalar generalized Verma modules attached to a maximal parabolic subalgebra $\mathfrak{p}$ is stated as follows.
\begin{theorem}[{\cite[Satz 3]{J}, \cite[Theorem 2.2.11]{M}}]
Let $\mathfrak{p}$ be a maximal parabolic subalgebra. Then the scalar generalized Verma module $M_\mathfrak{p}^\mathfrak{g}(\lambda)$ is irreducible if and only if\[\displaystyle{\sum_{\beta\in S_\lambda}Y(s_\beta(\lambda+\rho))=0}.\]
\end{theorem}
To use Jantzen's criterion, we need to determine whether $\displaystyle{\sum_{\beta\in S_\lambda}Y(s_\beta(\lambda+\rho))}$ is 0 or not. This is answered by the following proposition.
\begin{proposition}[{\cite[Proposition A.2.4]{Ku}}]
The sum $\displaystyle{\sum_{\beta\in S_\lambda}Y(s_\beta(\lambda+\rho))}$ is nonzero if and only if there is $\beta_0\in S_\lambda$ satisfying the following two conditions: (a) $Y(s_{\beta_0}(\lambda+\rho))\neq0$; and (b) there do not exist $\beta\in S_\lambda\setminus\{\beta_0\}$ and $\omega\in W_\mathfrak{l}$ of odd length such that $s_{\beta_0}(\lambda+\rho)=\omega s_\beta(\lambda+\rho)$.
\end{proposition}

\subsection{Special Lines}

Let us recall the construction of special lines in [\textbf{EHW}]. Henceforth $\mathfrak{g}$ is a finite-dimensional complex simple Lie algebra, and $\mathfrak{p}$ is a parabolic subalgebra of abelian type. Recall from Section 1 that $\zeta\in\mathfrak{h}^*$  is the unique weight such that $(\zeta,\alpha)=0$ for all $\alpha\in\Phi_\mathfrak{l}$ and $\langle\zeta,\gamma\rangle=1$.
\begin{definition}
If $\lambda_0\in\mathfrak{h}^*$ satisfies $(\lambda_0+\rho,\gamma)=0$, then $\lambda_0+z\zeta$ for $z\in\mathbb{C}$ is called a \textit{special line} in $\mathfrak{h}^*$.
\end{definition}
Obviously, every element $\lambda\in\mathfrak{h}^*$ can be expressed uniquely in the form $\lambda_0+z\zeta$ with $z\in\mathbb{C}$ and $(\lambda_0+\rho,\gamma)=0$. Hence, every element $\lambda\in\mathfrak{h}^*$ lies in some special line. In fact, given $\lambda\in\mathfrak{h}^*$, if $\lambda=\lambda_0+z\zeta$ with $z\in\mathbb{C}$ and $(\lambda_0+\rho,\gamma)=0$, then $z=\langle\lambda+\rho,\gamma\rangle$ and $\lambda_0=\lambda-z\zeta$. This shows uniqueness. Moreover, $\langle\lambda_0+\rho,\gamma\rangle=\langle\lambda-z\zeta+\rho,\gamma\rangle=\langle\lambda+\rho,\gamma\rangle-z=0$, and existence is showed.

The following result can be found in [\textbf{EHW}], which we state as a theorem here.
\begin{theorem}
For each special line $\lambda_0+z\zeta$, there exist three real numbers $A(\lambda_0)$, $B(\lambda_0)$, and $C(\lambda_0)$ with $C(\lambda_0)>0$, $A(\lambda_0)\leq B(\lambda_0)$, and $C(\lambda_0)^{-1}(B(\lambda_0)-A(\lambda_0))\in\mathbb{Z}_{\geq0}$ satisfying:
\begin{enumerate}[(1)]
\item If $z\in\mathbb{R}$ and $z<A(\lambda_0)$, then $M_\mathfrak{p}^\mathfrak{g}(\lambda_0+z\zeta)$ is simple.
\item If $z=A(\lambda_0)+iC(\lambda_0)$ for $0\leq i\leq C(\lambda_0)^{-1}(B(\lambda_0)-A(\lambda_0))$, then $M_\mathfrak{p}^\mathfrak{g}(\lambda_0+z\zeta)$ is reducible.
\end{enumerate}
\end{theorem}
\begin{proof}
See Theorem 2.4 and Proposition 3.1(b) in [\textbf{EHW}].
\end{proof}
We shall make full use of the values of $A(\lambda_0)$, $B(\lambda_0)$, and $C(\lambda_0)$ listed in [\textbf{EHW}] to do computations for the Lie algebra pairs $(\mathfrak{g},\mathfrak{p})$ with $\mathfrak{g}$ simple and $\mathfrak{p}$ a parabolic subalgebra of abelian type, case by case according to the Hermitian symmetric pairs.

\section{Reducibility of Scalar Generalized Verma Modules of Abelian Type}

In this section, we shall do computations case by case. According to [\textbf{W}], up to isomorphism there are seven complex Lie algebra pairs $(\mathfrak{g},\mathfrak{p})$ with $\mathfrak{g}$ simple and $\mathfrak{p}$ a parabolic subalgebra of abelian type, which correspond to the Hermitian symmetric pairs of compact type:
\begin{enumerate}[(i)]
\item $(SU(p+q),S(U(p)\times U(q)))$ for $p,q\geq1$,
\item $(Sp(n),U(n))$ for $n\geq2$,
\item $(SO(2n+1),SO(2)\times SO(2n-1))$ for $n\geq2$,
\item $(SO(2n),SO(2)\times SO(2n-2))$ for $n\geq2$,
\item $(SO(2n),U(n))$ for $n\geq2$,
\item $(E_{6(-78)},SO(2)\times SO(10))$,
\item $(E_{7(-133)},SO(2)\times E_{6(-78)})$.
\end{enumerate}
In each case, it is easily verified that the Minkowski sum $\Phi_\mathfrak{u}^++\Phi_\mathfrak{u}^+:=\{\alpha+\beta\mid\alpha,\beta\in\Phi_\mathfrak{u}^+\}\subseteq\mathfrak{h}^*$ does not intersect $\Phi_\mathfrak{u}^+$, whence $\mathfrak{u}_+$ is abelian.

We see below that the computations for the first two pairs are almost trivial because we do not need to use Jantzen's criterion, while the computations for the last two pairs are more involved because the root structures and Weyl groups of exceptional types are complicated. Moreover, for the pair $(SO(2n),U(n))$, we have to discuss separately according to the parity of $n$. Retain all the notations and settings in the previous two sections.

\subsection{$(SU(p+q),S(U(p)\times U(q)))$ for $p,q\geq1$}

Let $\mathfrak{g}=\mathfrak{sl}(p+q,\mathbb{C})$ and let $\mathfrak{p}=\mathfrak{l}+\mathfrak{u}_+$ be a parabolic subalgebra of abelian type with $\mathfrak{l}=\mathfrak{s}(\mathfrak{gl}(p,\mathbb{C})\oplus \mathfrak{gl}(q,\mathbb{C}))$. We may choose a Cartan subalgebra $\mathfrak{h}\subseteq \mathfrak{l}$ and a simple root system $\Delta=\{e_i-e_{i+1}\mid1\leq i\leq p+q-1\}$, such that $\mathfrak{p}$ is standard with respect to $(\mathfrak{h},\Delta)$. Then $\Delta_\mathfrak{l}=\{e_i-e_{i+1}\mid1\leq i\leq p+q-1,i\neq p\}$ and $\Phi_\mathfrak{u}^+=\{e_i-e_j\mid i\leq p,j>p\}$. Moreover, we have $\rho=\displaystyle{\sum_{i=1}^{p+q}(\frac{p+q+1}{2}-i)e_i}$.

If $M_\mathfrak{p}^\mathfrak{g}(\lambda)$ is of scalar type, an easy computation shows that $\lambda=\displaystyle{\frac{aq}{p+q}\sum_{i=1}^pe_i-\frac{ap}{p+q}\sum_{j=p+1}^{p+q}e_j}$ for some $a\in\mathbb{C}$.
\begin{lemma}
\label{a}
Suppose $\lambda=\displaystyle{\frac{aq}{p+q}\sum_{i=1}^pe_i-\frac{ap}{p+q}\sum_{j=p+1}^{p+q}e_j}$ for some $a\in\mathbb{C}$.
\begin{enumerate}[(1)]
\item If $a\notin2-p-q+\mathbb{Z}_{\geq0}$, then $M_\mathfrak{p}^\mathfrak{g}(\lambda)$ is simple.
\item If $a\in\mathbb{Z}_{\geq0}$, then $M_\mathfrak{p}^\mathfrak{g}(\lambda)$ is reducible.
\end{enumerate}
\end{lemma}
\begin{proof}
It is obvious that $\lambda+\rho$ is always $\Phi_\mathfrak{l}^+$-dominant integral. For $i\leq p$ and $j>p$, one computes that $\langle\lambda+\rho,e_i-e_j\rangle=j-i+a$. Now the conclusion follows from Theorem 2.1 immediately.
\end{proof}
The maximal root $\gamma$ in $\Phi^+$ is $e_1-e_n$, and it follows that $\zeta=\displaystyle{\frac{q}{p+q}\sum_{i=1}^pe_i-\frac{p}{p+q}\sum_{j=p+1}^{p+q}e_j}$. If we write $\lambda=\lambda_0+z\zeta$ in the special line, then we obtain $\lambda_0=\displaystyle{(\frac{q}{p+q}-q)\sum_{i=1}^pe_i+(p-\frac{p}{p+q})\sum_{j=p+1}^{p+q}e_j}$ and $a=z-p-q+1$. Now we may restate Lemma 3.1.
\begin{replemma}{a}
Suppose $\lambda=\displaystyle{\frac{aq}{p+q}\sum_{i=1}^pe_i+\frac{ap}{p+q}\sum_{j=p+1}^{p+q}e_j}$ for some $a\in\mathbb{C}$. Write $\lambda=\lambda_0+z\zeta$ in the special line.
\begin{enumerate}[(1)]
\item If $z\notin1+\mathbb{Z}_{\geq0}$, then $M_\mathfrak{p}^\mathfrak{g}(\lambda)$ is simple.
\item If $z\in p+q-1+\mathbb{Z}_{\geq0}$, then $M_\mathfrak{p}^\mathfrak{g}(\lambda)$ is reducible.
\end{enumerate}
\end{replemma}
\begin{proof}
Because $a=z-p-q+1$, the conclusion follows from Lemma 3.1 immediately.
\end{proof}
By Theorem 2.3, Lemma 7.3, and Theorem 7.4 in [\textbf{EHW}], one may check that $A(\lambda_0)=\max\{p,q\}$, $B(\lambda_0)=p+q-1$, and $C(\lambda_0)=1$ in this case.
\begin{theorem}
Suppose $\lambda=\displaystyle{\frac{aq}{p+q}\sum_{i=1}^pe_i+\frac{ap}{p+q}\sum_{j=p+1}^{p+q}e_j}$ for some $a\in\mathbb{C}$. Then $M_\mathfrak{p}^\mathfrak{g}(\lambda)$ is reducible if and only if $a\in1-\min\{p,q\}+\mathbb{Z}_{\geq0}$.
\end{theorem}
\begin{proof}
Write $\lambda=\lambda_0+z\zeta$ in the special line, and then Lemma 3.1$'$(1) and Theorem 2.6(1) show that $M_\mathfrak{p}^\mathfrak{g}(\lambda)$ is reducible only if $z\in\max\{p,q\}+\mathbb{Z}_{\geq0}$. On the other hand, Lemma 3.1$'$(2) and Theorem 2.6(2) show that if $z\in\max\{p,q\}+\mathbb{Z}_{\geq0}$, then $M_\mathfrak{p}^\mathfrak{g}(\lambda)$ is reducible. It follows that $M_\mathfrak{p}^\mathfrak{g}(\lambda)$ is reducible if and only if $z\in\max\{p,q\}+\mathbb{Z}_{\geq0}$, which is equivalent to $a\in\max\{p,q\}-p-q+1+\mathbb{Z}_{\geq0}=1-\min\{p,q\}+\mathbb{Z}_{\geq0}$.
\end{proof}

\subsection{$(Sp(n),U(n))$ for $n\geq2$}

Let $\mathfrak{g}=\mathfrak{sp}(2n,\mathbb{C})$ and let $\mathfrak{p}=\mathfrak{l}+\mathfrak{u}_+$ be a parabolic subalgebra of abelian type with $\mathfrak{l}=\mathfrak{gl}(n,\mathbb{C})$. We may choose a Cartan subalgebra $\mathfrak{h}\subseteq \mathfrak{l}$ and a simple root system $\Delta=\{e_i-e_{i+1}\mid1\leq i\leq n-1\}\cup\{2e_n\}$, such that $\mathfrak{p}$ is standard with respect to $(\mathfrak{h},\Delta)$. Then $\Delta_\mathfrak{l}=\{e_i-e_{i+1}\mid1\leq i\leq n-1\}$ and $\Phi_\mathfrak{u}^+=\{e_i+e_j\mid 1\leq i<j\leq n\}\cup\{2e_k\mid1\leq k\leq n\}$. Moreover, we have $\rho=\displaystyle{\sum_{i=1}^n(n-i+1)e_i}$.

If $M_\mathfrak{p}^\mathfrak{g}(\lambda)$ is of scalar type, an easy computation shows that $\lambda=\displaystyle{a\sum_{i=1}^ne_i}$ for some $a\in\mathbb{C}$.
\begin{lemma}
\label{b}
Suppose $\lambda=\displaystyle{a\sum_{i=1}^ne_i}$ for some $a\in\mathbb{C}$.
\begin{enumerate}[(1)]
\item If $a\notin1-n+\frac{1}{2}\mathbb{Z}_{\geq0}$, then $M_\mathfrak{p}^\mathfrak{g}(\lambda)$ is simple.
\item If $a\in\frac{1}{2}\mathbb{Z}_{\geq0}$, then $M_\mathfrak{p}^\mathfrak{g}(\lambda)$ is reducible.
\end{enumerate}
\end{lemma}
\begin{proof}
It is obvious that $\lambda+\rho$ is always $\Phi_\mathfrak{l}^+$-dominant integral. For $1\leq i<j\leq n$ and $1\leq k\leq n$, one computes that $\langle\lambda+\rho,e_i+e_j\rangle=2a+2n-i-j+2$ and $\langle\lambda+\rho,2e_k\rangle=a+n-k+1$. Now the conclusion follows from Theorem 2.1 immediately.
\end{proof}
The maximal root $\gamma$ in $\Phi^+$ is $2e_1$, and it follows that $\zeta=\displaystyle{\sum_{i=1}^ne_i}$. If we write $\lambda=\lambda_0+z\zeta$ in the special line, then we obtain $\lambda_0=\displaystyle{-n\sum_{i=1}^ne_i}$ and $a=z-n$. Now we may restate Lemma 3.3.
\begin{replemma}{b}
Suppose $\lambda=\displaystyle{a\sum_{i=1}^ne_i}$ for some $a\in\mathbb{C}$. Write $\lambda=\lambda_0+z\zeta$ in the special line.
\begin{enumerate}[(1)]
\item If $z\notin1+\frac{1}{2}\mathbb{Z}_{\geq0}$, then $M_\mathfrak{p}^\mathfrak{g}(\lambda)$ is simple.
\item If $z\in n+\frac{1}{2}\mathbb{Z}_{\geq0}$, then $M_\mathfrak{p}^\mathfrak{g}(\lambda)$ is reducible.
\end{enumerate}
\end{replemma}
\begin{proof}
Because $a=z-n$, the conclusion follows from Lemma 3.3 immediately.
\end{proof}
By Theorem 2.3, Lemma 8.3, and Theorem 8.4 in [\textbf{EHW}], one may check that $A(\lambda_0)=\frac{n+1}{2}$, $B(\lambda_0)=n$, and $C(\lambda_0)=\frac{1}{2}$ in this case. Unlike the case in [\textbf{EHW}, Section 7], it is not trivial to obtain these three numbers in this case. In fact, we have to show $q=r$ in Lemma 8.3 and Theorem 8.4 in [\textbf{EHW}]. According to the construction of $R(\lambda_0)$ and $Q(\lambda_0)$ under [\textbf{EHW}, Theorem 2.4], we have $R(\lambda_0)=Q(\lambda_0)$ which is the full root system of $\mathfrak{sp}(n)$, and then [\textbf{EHW}, Equation (8.1)] shows $q=r=n$ in our case. The three numbers $A(\lambda_0)$, $B(\lambda_0)$, and $C(\lambda_0)$ in the cases from Section 3.3 to Section 3.7 can be obtained by similar computations, but we shall only mention the three numbers and the details are left to the reader.
\begin{theorem}
Suppose $\lambda=\displaystyle{a\sum_{i=1}^ne_i}$ for some $a\in\mathbb{C}$. Then $M_\mathfrak{p}^\mathfrak{g}(\lambda)$ is reducible if and only if $a\in\frac{1-n}{2}+\frac{1}{2}\mathbb{Z}_{\geq0}$.
\end{theorem}
\begin{proof}
Write $\lambda=\lambda_0+z\zeta$ in the special line, and then Lemma 3.3$'$(1) and Theorem 2.6(1) show that $M_\mathfrak{p}^\mathfrak{g}(\lambda)$ is reducible only if $z\in\frac{n+1}{2}+\frac{1}{2}\mathbb{Z}_{\geq0}$. On the other hand, Lemma 3.3$'$(2) and Theorem 2.6(2) show that if $z\in\frac{n+1}{2}+\frac{1}{2}\mathbb{Z}_{\geq0}$, then $M_\mathfrak{p}^\mathfrak{g}(\lambda)$ is reducible. It follows that $M_\mathfrak{p}^\mathfrak{g}(\lambda)$ is reducible if and only if $z\in\frac{n+1}{2}+\frac{1}{2}\mathbb{Z}_{\geq0}$, which is equivalent to $a\in\frac{1-n}{2}+\frac{1}{2}\mathbb{Z}_{\geq0}$.
\end{proof}

\subsection{$(SO(2n+1),SO(2)\times SO(2n-1))$ for $n\geq2$}

Let $\mathfrak{g}=\mathfrak{so}(2n+1,\mathbb{C})$ and let $\mathfrak{p}=\mathfrak{l}+\mathfrak{u}_+$ be a parabolic subalgebra of abelian type with $\mathfrak{l}=\mathfrak{so}(2,\mathbb{C})\oplus\mathfrak{so}(2n-1,\mathbb{C})$. We may choose a Cartan subalgebra $\mathfrak{h}\subseteq \mathfrak{l}$ and a simple root system $\Delta=\{e_i-e_{i+1}\mid1\leq i\leq n-1\}\cup\{e_n\}$, such that $\mathfrak{p}$ is standard with respect to $(\mathfrak{h},\Delta)$. Then $\Delta_\mathfrak{l}=\{e_i-e_{i+1}\mid2\leq i\leq n-1\}\cup\{e_n\}$ and $\Phi_\mathfrak{u}^+=\{e_1\pm e_j\mid2\leq j\leq n\}\cup\{e_1\}$. Moreover, we have $\rho=\displaystyle{\sum_{i=1}^n(n-i+\frac{1}{2})e_i}$.

If $M_\mathfrak{p}^\mathfrak{g}(\lambda)$ is of scalar type, an easy computation shows that $\lambda=ae_1$ for some $a\in\mathbb{C}$.
\begin{lemma}
\label{d}
Suppose $\lambda=ae_1$ for some $a\in\mathbb{C}$.
\begin{enumerate}[(1)]
\item If $a\notin(3-2n+\mathbb{Z}_{\geq0})\cup(\frac{3}{2}-n+\mathbb{Z}_{\geq0})$, then $M_\mathfrak{p}^\mathfrak{g}(\lambda)$ is simple.
\item If $a\in\mathbb{Z}_{\geq0}\cup(\frac{3}{2}-n+\mathbb{Z}_{\geq0})$, then $M_\mathfrak{p}^\mathfrak{g}(\lambda)$ is reducible.
\end{enumerate}
\end{lemma}
\begin{proof}
It is obvious that $\lambda+\rho$ is always $\Phi_\mathfrak{l}^+$-dominant integral. For $2\leq j\leq n$, one computes that $\langle\lambda+\rho,e_1+e_j\rangle=a+2n-j$ and $\langle\lambda+\rho,e_1-e_j\rangle=a+j-1$. Moreover, $\langle\lambda+\rho,e_1\rangle=2a+2n-1$. Now the conclusion follows from Theorem 2.1 immediately.
\end{proof}
The maximal root $\gamma$ in $\Phi^+$ is $e_1+e_2$, and it follows that $\zeta=e_1$. If we write $\lambda=\lambda_0+z\zeta$ in the special line, then we obtain $\lambda_0=(2-2n)e_1$ and $a=z-2n+2$. Now we may restate Lemma 3.5.
\begin{replemma}{d}
Suppose $\lambda=ae_1$ for some $a\in\mathbb{C}$. Write $\lambda=\lambda_0+z\zeta$ in the special line.
\begin{enumerate}[(1)]
\item If $z\notin(1+\mathbb{Z}_{\geq0})\cup(n-\frac{1}{2}+\mathbb{Z}_{\geq0})$, then $M_\mathfrak{p}^\mathfrak{g}(\lambda)$ is simple.
\item If $z\in(n-\frac{1}{2}+\mathbb{Z}_{\geq0})\cup(2n-2+\mathbb{Z}_{\geq0})$, then $M_\mathfrak{p}^\mathfrak{g}(\lambda)$ is reducible.
\end{enumerate}
\end{replemma}
\begin{proof}
Because $a=z-2n+2$, the conclusion follows from Lemma 3.5 immediately.
\end{proof}
By Theorem 2.3, Lemma 11.3, and Theorem 11.4 in [\textbf{EHW}], one may check that $A(\lambda_0)=n-\frac{1}{2}$, $B(\lambda_0)=2n-2$, and $C(\lambda_0)=n-\frac{3}{2}$ in this case.
\begin{proposition}
Suppose $\lambda=ae_1$ for some $a\in\mathbb{C}$. Write $\lambda=\lambda_0+z\zeta$ in the special line. If $z\in\mathbb{Z}$ and $A(\lambda_0)<z<B(\lambda_0)$, then $M_\mathfrak{p}^\mathfrak{g}(\lambda)$ is simple.
\end{proposition}
\begin{proof}
Because $a=z-2n+2$, in fact we already computed that $\langle\lambda+\rho,e_1+e_j\rangle=z-j+2$, $\langle\lambda+\rho,e_1-e_j\rangle=z-2n+j+1$ for $2\leq j\leq n$, and $\langle\lambda+\rho,e_1\rangle=2z-2n+3$. If $z\in\mathbb{Z}$ and $A(\lambda_0)=n-\frac{1}{2}<z<B(\lambda_0)=2n-2$, then\[S_\lambda=\{e_1+e_j\mid2\leq j\leq n\}\cup\{e_1-e_j\mid2n-z-1<j\leq n\}\cup\{e_1\}.\]Since $\lambda+\rho=(z-n+\frac{3}{2})e_1+\displaystyle{\sum_{i=2}^n}(n-i+\frac{1}{2})e_i$, we have that $(s_{e_1+e_j}(\lambda+\rho),e_{2n-z-1}+e_j)=0$ for $2\leq j\leq n$ and $j\neq2n-z-1$. It follows if $2\leq j\leq n$ and $j\neq2n-z-1$, then $s_{e_1+e_j}(\lambda+\rho)$ is $\Phi_\mathfrak{l}$-singular, and hence $Y(s_{e_1+e_j}(\lambda+\rho))=0$ by Proposition 2.2(1). Similarly, we have $(s_{e_1-e_j}(\lambda+\rho),e_{2n-z-1}-e_j)=0$ for $2n-z-1<j\leq n$, so $Y(s_{e_1-e_j}(\lambda+\rho))=0$ for $2n-z-1<j\leq n$. On the other hand, it is easy to check that $s_{e_1+e_{2n-z-1}}(\lambda+\rho)$ and $s_{e_1}(\lambda+\rho)$ are $\Phi_\mathfrak{l}$-regular integral, and hence $Y(s_{e_1+e_{2n-z-1}}(\lambda+\rho))\neq0$ and $Y(s_{e_1}(\lambda+\rho))\neq0$ by Proposition 2.2(1). Moreover, because $s_{e_1+e_{2n-z-1}}(\lambda+\rho)=s_{e_{2n-z-1}}s_{e_1}(\lambda+\rho)$ and $s_{e_{2n-z-1}}\in W_\mathfrak{l}$ is of odd length, it follows from Proposition 2.2(2) that $Y(s_{e_1+e_{2n-z-1}}(\lambda+\rho))+Y(s_{e_1}(\lambda+\rho))=0$. Now\[\displaystyle{\sum_{\beta\in S_\lambda}Y(s_\beta(\lambda+\rho))}=Y(s_{e_1+e_{2n-z-1}}(\lambda+\rho))+Y(s_{e_1}(\lambda+\rho))=0.\]By Theorem 2.3, $M_\mathfrak{p}^\mathfrak{g}(\lambda)$ is simple.
\end{proof}
\begin{theorem}
Suppose $\lambda=ae_1$ for some $a\in\mathbb{C}$. Then $M_\mathfrak{p}^\mathfrak{g}(\lambda)$ is reducible if and only if $a\in\mathbb{Z}_{\geq0}\cup(\frac{3}{2}-n+\mathbb{Z}_{\geq0})$.
\end{theorem}
\begin{proof}
Write $\lambda=\lambda_0+z\zeta$ in the special line, and then Lemma 3.5$'$(1), Proposition 3.6, and Theorem 2.6(1) show that $M_\mathfrak{p}^\mathfrak{g}(\lambda)$ is reducible only if $z\in(n-\frac{1}{2}+\mathbb{Z}_{\geq0})\cup(2n-2+\mathbb{Z}_{\geq0})$. On the other hand, Lemma 3.5$'$(2) shows that if $z\in(n-\frac{1}{2}+\mathbb{Z}_{\geq0})\cup(2n-2+\mathbb{Z}_{\geq0})$, then $M_\mathfrak{p}^\mathfrak{g}(\lambda)$ is reducible. It follows that $M_\mathfrak{p}^\mathfrak{g}(\lambda)$ is reducible if and only if $z\in(n-\frac{1}{2}+\mathbb{Z}_{\geq0})\cup(2n-2+\mathbb{Z}_{\geq0})$, which is equivalent to $a\in\mathbb{Z}_{\geq0}\cup(\frac{3}{2}-n+\mathbb{Z}_{\geq0})$.
\end{proof}

\subsection{$(SO(2n),SO(2)\times SO(2n-2))$ for $n\geq2$}

Let $\mathfrak{g}=\mathfrak{so}(2n,\mathbb{C})$ and let $\mathfrak{p}=\mathfrak{l}+\mathfrak{u}_+$ be a parabolic subalgebra of abelian type with $\mathfrak{l}=\mathfrak{so}(2,\mathbb{C})\oplus\mathfrak{so}(2n-2,\mathbb{C})$. We may choose a Cartan subalgebra $\mathfrak{h}\subseteq \mathfrak{l}$ and a simple root system $\Delta=\{e_i-e_{i+1}\mid1\leq i\leq n-1\}\cup\{e_{n-1}+e_n\}$, such that $\mathfrak{p}$ is standard with respect to $(\mathfrak{h},\Delta)$. Then $\Delta_\mathfrak{l}=\{e_i-e_{i+1}\mid2\leq i\leq n-1\}\cup\{e_{n-1}+e_n\}$ and $\Phi_\mathfrak{u}^+=\{e_1\pm e_j\mid2\leq j\leq n\}$. Moreover, we have $\rho=\displaystyle{\sum_{i=1}^n(n-i)e_i}$.

If $M_\mathfrak{p}^\mathfrak{g}(\lambda)$ is of scalar type, an easy computation shows that $\lambda=ae_1$ for some $a\in\mathbb{C}$.
\begin{lemma}
\label{c}
Suppose $\lambda=ae_1$ for some $a\in\mathbb{C}$.
\begin{enumerate}[(1)]
\item If $a\notin4-2n+\mathbb{Z}_{\geq0}$, then $M_\mathfrak{p}^\mathfrak{g}(\lambda)$ is simple.
\item If $a\in\mathbb{Z}_{\geq0}$, then $M_\mathfrak{p}^\mathfrak{g}(\lambda)$ is reducible.
\end{enumerate}
\end{lemma}
\begin{proof}
It is obvious that $\lambda+\rho$ is always $\Phi_\mathfrak{l}^+$-dominant integral. For $2\leq j\leq n$, one computes that $\langle\lambda+\rho,e_1+e_j\rangle=a+2n-j-1$ and $\langle\lambda+\rho,e_1-e_j\rangle=a+j-1$. Now the conclusion follows from Theorem 2.1 immediately.
\end{proof}
The maximal root $\gamma$ in $\Phi^+$ is $e_1+e_2$, and it follows that $\zeta=e_1$. If we write $\lambda=\lambda_0+z\zeta$ in the special line, then we obtain $\lambda_0=(3-2n)e_1$ and $a=z-2n+3$. Now we may restate Lemma 3.8.
\begin{replemma}{c}
Suppose $\lambda=ae_1$ for some $a\in\mathbb{C}$. Write $\lambda=\lambda_0+z\zeta$ in the special line.
\begin{enumerate}[(1)]
\item If $z\notin1+\mathbb{Z}_{\geq0}$, then $M_\mathfrak{p}^\mathfrak{g}(\lambda)$ is simple.
\item If $z\in2n-3+\mathbb{Z}_{\geq0}$, then $M_\mathfrak{p}^\mathfrak{g}(\lambda)$ is reducible.
\end{enumerate}
\end{replemma}
\begin{proof}
Because $a=z-2n+3$, the conclusion follows from Lemma 3.8 immediately.
\end{proof}
By Theorem 2.3, Lemma 10.3, and Theorem 10.4 in [\textbf{EHW}], one may check that $A(\lambda_0)=n-1$, $B(\lambda_0)=2n-3$, and $C(\lambda_0)=n-2$ in this case.
\begin{proposition}
Suppose $\lambda=ae_1$ for some $a\in\mathbb{C}$. Write $\lambda=\lambda_0+z\zeta$ in the special line. If $z\in\mathbb{Z}$ and $A(\lambda_0)<z<B(\lambda_0)$, then $M_\mathfrak{p}^\mathfrak{g}(\lambda)$ is reducible.
\end{proposition}
\begin{proof}
Because $a=z-2n+3$, in fact we already computed that $\langle\lambda+\rho,e_1+e_j\rangle=z-j+2$ and $\langle\lambda+\rho,e_1-e_j\rangle=z-2n+j+2$ for $2\leq j\leq n$. If $z\in\mathbb{Z}$ and $A(\lambda_0)=n-1<z<B(\lambda_0)=2n-3$, then\[S_\lambda=\{e_1+e_j\mid2\leq j\leq n\}\cup\{e_1-e_j\mid2n-z-2<j\leq n\}.\]Since $\lambda+\rho=(z-n+2)e_1+\displaystyle{\sum_{i=2}^n}(n-i)e_i$, we have that $(s_{e_1+e_j}(\lambda+\rho),e_{2n-z-2}+e_j)=0$ for $2\leq j\leq n$ and $j\neq2n-z-2$. It follows that if $2\leq j\leq n$ and $j\neq2n-z-2$, then $s_{e_1+e_j}(\lambda+\rho)$ is $\Phi_\mathfrak{l}$-singular, and hence $Y(s_{e_1+e_j}(\lambda+\rho))=0$ by Proposition 2.2(1). Similarly, if $2n-z-2<j\leq n$, then $(s_{e_1-e_j}(\lambda+\rho),e_{2n-z-2}-e_j)=0$, and hence $Y(s_{e_1-e_j}(\lambda+\rho))=0$. On the other hand, it is easy to check that $s_{e_1+e_{2n-z-2}}(\lambda+\rho)$ is $\Phi_\mathfrak{l}$-regular integral, and hence $Y(s_{e_1+e_{2n-z-2}}(\lambda+\rho))\neq0$ by Proposition 2.2(1). Now\[\displaystyle{\sum_{\beta\in S_\lambda}Y(s_\beta(\lambda+\rho))}=Y(s_{e_1+e_{2n-z-2}(\lambda+\rho)})\neq0.\]By Theorem 2.3, $M_\mathfrak{p}^\mathfrak{g}(\lambda)$ is reducible.
\end{proof}
\begin{theorem}
Suppose $\lambda=ae_1$ for some $a\in\mathbb{C}$. Then $M_\mathfrak{p}^\mathfrak{g}(\lambda)$ is reducible if and only if $a\in2-n+\mathbb{Z}_{\geq0}$.
\end{theorem}
\begin{proof}
Write $\lambda=\lambda_0+z\zeta$ in the special line, and then Lemma 3.8$'$(1) and Theorem 2.6(1) show that $M_\mathfrak{p}^\mathfrak{g}(\lambda)$ is reducible only if $z\in n-1+\mathbb{Z}_{\geq0}$. On the other hand, Lemma 3.8$'$(2), Proposition 3.9, and Theorem 2.6(2) show that if $z\in n-1+\mathbb{Z}_{\geq0}$, then $M_\mathfrak{p}^\mathfrak{g}(\lambda)$ is reducible. It follows that $M_\mathfrak{p}^\mathfrak{g}(\lambda)$ is reducible if and only if $z\in n-1+\mathbb{Z}_{\geq0}$, which is equivalent to $a\in2-n+\mathbb{Z}_{\geq0}$.
\end{proof}

\subsection{$(SO(2n),U(n))$ for $n\geq2$}

Let $\mathfrak{g}=\mathfrak{so}(2n,\mathbb{C})$ and let $\mathfrak{p}=\mathfrak{l}+\mathfrak{u}_+$ be a parabolic subalgebra of abelian type with $\mathfrak{l}=\mathfrak{gl}(n,\mathbb{C})$. We may choose a Cartan subalgebra $\mathfrak{h}\subseteq \mathfrak{l}$ and a simple root system $\Delta=\{e_i-e_{i+1}\mid1\leq i\leq n-1\}\cup\{e_{n-1}+e_n\}$, such that $\mathfrak{p}$ is standard with respect to $(\mathfrak{h},\Delta)$. Then $\Delta_\mathfrak{l}=\{e_i-e_{i+1}\mid1\leq i\leq n-1\}$ and $\Phi_\mathfrak{u}^+=\{e_i+e_j\mid1\leq i<j\leq n\}$. Moreover, we have $\rho=\displaystyle{\sum_{i=1}^n(n-i)e_i}$.

If $M_\mathfrak{p}^\mathfrak{g}(\lambda)$ is of scalar type, an easy computation shows that $\lambda=\displaystyle{a\sum_{i=1}^ne_i}$ for some $a\in\mathbb{C}$.
\begin{lemma}
\label{e}
Suppose $\lambda=\displaystyle{a\sum_{i=1}^ne_i}$ for some $a\in\mathbb{C}$.
\begin{enumerate}[(1)]
\item If $a\notin2-n+\frac{1}{2}\mathbb{Z}_{\geq0}$, then $M_\mathfrak{p}^\mathfrak{g}(\lambda)$ is simple.
\item If $a\in\frac{1}{2}\mathbb{Z}_{\geq0}$, then $M_\mathfrak{p}^\mathfrak{g}(\lambda)$ is reducible.
\end{enumerate}
\end{lemma}
\begin{proof}
It is obvious that $\lambda+\rho$ is always $\Phi_\mathfrak{l}^+$-dominant integral. For $1\leq i<j\leq n$, one computes that $\langle\lambda+\rho,e_i+e_j\rangle=2a+2n-i-j$. Now the conclusion follows from Theorem 2.1 immediately.
\end{proof}
The maximal root $\gamma$ in $\Phi^+$ is $e_1+e_2$, and it follows that $\zeta=\displaystyle{\frac{1}{2}\sum_{i=1}^ne_i}$. If we write $\lambda=\lambda_0+z\zeta$ in the special line, then we obtain $\lambda_0=\displaystyle{(\frac{3}{2}-n)\sum_{i=1}^ne_i}$ and $z=2a+2n-3$. Now we may restate Lemma 3.11.
\begin{replemma}{e}
Suppose $\lambda=\displaystyle{a\sum_{i=1}^ne_i}$ for some $a\in\mathbb{C}$. Write $\lambda=\lambda_0+z\zeta$ in the special line.
\begin{enumerate}[(1)]
\item If $z\notin1+\mathbb{Z}_{\geq0}$, then $M_\mathfrak{p}^\mathfrak{g}(\lambda)$ is simple.
\item If $z\in2n-3+\mathbb{Z}_{\geq0}$, then $M_\mathfrak{p}^\mathfrak{g}(\lambda)$ is reducible.
\end{enumerate}
\end{replemma}
\begin{proof}
Because $z=2a+2n-3$, the conclusion follows from Lemma 3.11 immediately.
\end{proof}
By Theorem 2.3, Lemma 9.3, and Theorem 9.4 in [\textbf{EHW}], one may check that $B(\lambda_0)=2n-3$ and $C(\lambda_0)=2$ in this case. The value of $A(\lambda_0)$ depends on the parity of $n$. If $n$ is even, and $A(\lambda_0)=n-1$; and if $n$ is odd, $A(\lambda_0)=n$.
\begin{proposition}
Suppose $\lambda=\displaystyle{a\sum_{i=1}^ne_i}$ for some $a\in\mathbb{C}$. Write $\lambda=\lambda_0+z\zeta$ in the special line. If $z\in\mathbb{Z}$ and $A(\lambda_0)<z<B(\lambda_0)$, then $M_\mathfrak{p}^\mathfrak{g}(\lambda)$ is reducible.
\end{proposition}
\begin{proof}
Let us assume that $n$ is even, and then $A(\lambda_0)=n-1$. Therefore Proposition 3.1(b), Lemma 9.3, and Theorem 9.4 in [\textbf{EHW}] imply that if $z=A(\lambda_0)+2k$ for $k\in\mathbb{Z}_{\geq0}$ and $z\leq B(\lambda_0)$, then $M_\mathfrak{p}^\mathfrak{g}(\lambda)$ is reducible. Thus we only need to check $z=A(\lambda_0)+2k+1=n+2k$ for $k\in\mathbb{Z}_{\geq0}$ and $0\leq k\leq\frac{n}{2}-2$.

Because $z=2a+2n-3$, in fact we already computed that $\langle\lambda+\rho,e_i+e_j\rangle=n+2k-i-j+3$ for $1\leq i<j\leq n$. If $k\in\mathbb{Z}_{\geq0}$ and $0\leq k\leq\frac{n}{2}-2$, then\[S_\lambda=\{e_i+e_j\mid i+j<n+2k+3\}.\]Now consider $\lambda+\rho=\displaystyle{\sum_{m=1}^n}(\frac{n+3}{2}+k-m)e_m$. If $j\geq2k+3$, since $i+j<n+2k+3$, then $(s_{e_i+e_j}(\lambda+\rho),e_i-e_{n+2k+3-j})=0$. It follows that $s_{e_i+e_j}(\lambda+\rho)$ is $\Phi_\mathfrak{l}$-singular, so $Y(s_{e_i+e_j}(\lambda+\rho))=0$ by Proposition 2.2(1). Hence, we only need to consider $Y(s_{e_i+e_j}(\lambda+\rho))$ for $i<j<2k+3$. It is obvious that $s_{e_1+e_2}(\lambda+\rho)$ is $\Phi_\mathfrak{l}$-regular integral, and then $Y(s_{e_1+e_2}(\lambda+\rho))\neq0$ by Proposition 2.2(1). For $i<j<2k+3$ and $e_i+e_j\neq e_1+e_2$, we claim that if there exist $e_i+e_j\in S_\lambda$ and $\omega\in W_\mathfrak{l}$ such that $s_{e_1+e_2}(\lambda+\rho)=\omega s_{e_i+e_j}(\lambda+\rho)$, then $i=1$ and $j=2$, so $\displaystyle{\sum_{\beta\in S_\lambda}Y(s_\beta(\lambda+\rho))}\neq0$ by Proposition 2.4, and hence the conclusion holds by Theorem 2.3. To show this claim, let $c_m:=\frac{n+3}{2}+k-m$ denote the coefficient of $e_m$ in $\lambda+\rho$, i.e., let $\lambda+\rho=\displaystyle{\sum_{m=1}^nc_me_m}$. It is easy to see that $c_1>c_2>|c_m|$ for $3\leq m\leq n$, and $s_{e_1+e_2}(\lambda+\rho)=-c_2e_1-c_1e_2+\displaystyle{\sum_{m=3}^nc_me_m}$. Because $W_\mathfrak{l}\cong S_n$ is a symmetric group, each $\omega\in W_\mathfrak{l}$ is a permutation of the basis vectors $e_m$ for $1\leq m\leq n$. It follows that $i=1$ and $j=2$ if $s_{e_1+e_2}(\lambda+\rho)=\omega s_{e_i+e_j}(\lambda+\rho)$ for some $\omega\in W_\mathfrak{l}$, since $-c_1$ and $-c_2$ must appear as coefficients of two basis vectors in $s_{e_i+e_j}(\lambda+\rho)$ and $-c_1<-c_2<-|c_m|$ for $3\leq m\leq n$. The claim holds, and the conclusion for $n$ even is proved.

The proof for $n$ odd is parallel. If $n$ is odd, then $A(\lambda_0)=n$, and we need to check $z=n+2k+1$ for $k\in\mathbb{Z}_{\geq0}$ and $0\leq k\leq\frac{n-5}{2}$. It is computed that\[S_\lambda=\{e_i+e_j\mid i+j<n+2k+4\}.\]Finally, verify similarly to the above analysis that $Y(s_{e_1+e_2}(\lambda+\rho))\neq0$ cannot be cancelled out in $\displaystyle{\sum_{\beta\in S_\lambda}Y(s_\beta(\lambda+\rho))}$, and apply Theorem 2.3 to conclude the proof.
\end{proof}
\begin{theorem}
Suppose $\lambda=\displaystyle{a\sum_{i=1}^ne_i}$ for some $a\in\mathbb{C}$. Then $M_\mathfrak{p}^\mathfrak{g}(\lambda)$ is reducible if and only if \[a\in\begin{cases}1-\frac{n}{2}+\frac{1}{2}\mathbb{Z}_{\geq0},&n\textrm{ even},\\\frac{3-n}{2}+\frac{1}{2}\mathbb{Z}_{\geq0},&n\textrm{ odd}.\end{cases}\]
\end{theorem}
\begin{proof}
Assume that $n$ is even. Write $\lambda=\lambda_0+z\zeta$ in the special line, and then Lemma 3.11$'$(1) and Theorem 2.6(1) show that $M_\mathfrak{p}^\mathfrak{g}(\lambda)$ is reducible only if $z\in n-1+\mathbb{Z}_{\geq0}$. On the other hand, Lemma 3.11$'$(2), Proposition 3.12, and Theorem 2.6(2) show that if $z\in n-1+\mathbb{Z}_{\geq0}$, then $M_\mathfrak{p}^\mathfrak{g}(\lambda)$ is reducible. It follows that $M_\mathfrak{p}^\mathfrak{g}(\lambda)$ is reducible if and only if $z\in n-1+\mathbb{Z}_{\geq0}$, which is equivalent to $a\in1-\frac{n}{2}+\frac{1}{2}\mathbb{Z}_{\geq0}$. The proof for $n$ odd is parallel.
\end{proof}

\subsection{$(E_{6(-78)},SO(2)\times SO(10))$}

Let $\mathfrak{g}=\mathfrak{e}_6$ and let $\mathfrak{p}=\mathfrak{l}+\mathfrak{u}_+$ be a parabolic subalgebra of abelian type with $\mathfrak{l}=\mathfrak{so}(2,\mathbb{C})\oplus\mathfrak{so}(10,\mathbb{C})$. We may choose a Cartan subalgebra $\mathfrak{h}\subseteq \mathfrak{l}$ and a simple root system $\Delta=\{\alpha_i\mid1\leq i\leq6\}$ given by the Dynkin diagram of Figure 1 such that $\mathfrak{p}$ is standard with respect to $(\mathfrak{h},\Delta)$. Then $\Delta_\mathfrak{l}=\{\alpha_i\mid2\leq i\leq6\}$.
\begin{figure}[H]
\centering \scalebox{0.7}{\includegraphics{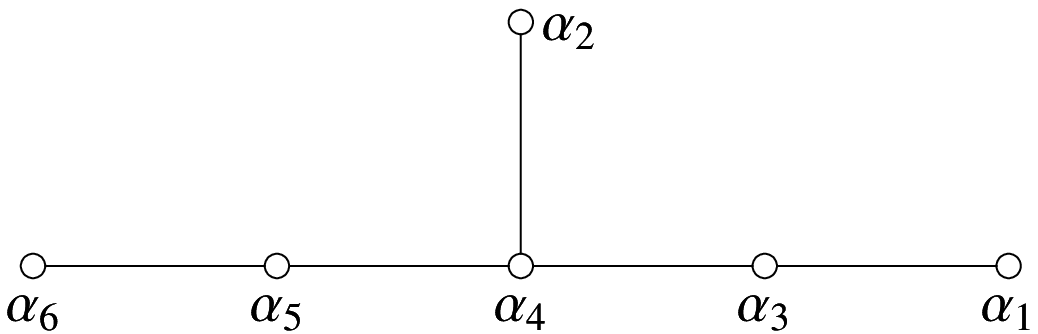}}
\caption{Dynkin diagram of $\mathfrak{e}_6$.}
\end{figure}
Embed $\mathfrak{h}_\mathbb{R}^*$, the $\mathbb{R}$-span of the simple roots, into the subspace $V_6:=\{v\in\mathbb{R}^8\mid(v,e_6-e_7)=(v,e_7+e_8)=0\}$ of $\mathbb{R}^8$, and let $\alpha_1=\frac{1}{2}(e_1-e_2-e_3-e_4-e_5-e_6-e_7+e_8)$, $\alpha_2=e_1+e_2$, $\alpha_i=e_{i-1}-e_{i-2}$ for $3\leq i\leq6$. Now given $\mathrm{v}\in\mathbb{Z}^5$, define
\begin{equation}
\label{i}
\alpha_\pm(\mathrm{v}):=\frac{1}{2}(\displaystyle{\sum_{i=1}^5}(-1)^{v(i)}e_i\pm e_6-e_7+e_8).
\end{equation}
Then\[\Phi_\mathfrak{l}^+=\{e_j\pm e_i\mid1\leq i<j\leq5\}\]and\[\Phi_\mathfrak{u}^+=\{\alpha_-(\mathrm{v})\mid\displaystyle{\sum_{i=1}^5}v(i)\textrm{ even}\}.\]Moreover, we have $\rho=e_2+2e_3+3e_4+4e_5-4e_6-4e_7+4e_8$.

If $M_\mathfrak{p}^\mathfrak{g}(\lambda)$ is of scalar type, an easy computation shows that $\lambda=-ae_6-ae_7+ae_8$ for some $a\in\mathbb{C}$.
\begin{lemma}
\label{f}
Suppose $\lambda=-ae_6-ae_7+ae_8$ for some $a\in\mathbb{C}$.
\begin{enumerate}[(1)]
\item If $a\notin-\frac{20}{3}+\frac{2}{3}\mathbb{Z}_{\geq0}$, then $M_\mathfrak{p}^\mathfrak{g}(\lambda)$ is simple.
\item If $a\in\frac{2}{3}\mathbb{Z}_{\geq0}$, then $M_\mathfrak{p}^\mathfrak{g}(\lambda)$ is reducible.
\end{enumerate}
\end{lemma}
\begin{proof}
It is obvious that $\lambda+\rho$ is always $\Phi_\mathfrak{l}^+$-dominant integral. For $\alpha_-(\mathrm{v})\in\Phi_\mathfrak{u}^+$, one computes that $\langle\lambda+\rho,\alpha_-(\mathrm{v})\rangle=\frac{1}{2}(3a+(-1)^{v(2)}+2(-1)^{v(3)}+3(-1)^{v(4)}+4(-1)^{v(5)})+6$. Now the conclusion follows from Theorem 2.1 immediately.
\end{proof}
The maximal root $\gamma$ in $\Phi^+$ is $\frac{1}{2}(e_1+e_2+e_3+e_4+e_5-e_6-e_7+e_8)$, and it follows that $\zeta=\frac{2}{3}(-e_6-e_7+e_8)$. If we write $\lambda=\lambda_0+z\zeta$ in the special line, then we obtain $\lambda_0=\frac{22}{3}e_6+\frac{22}{3}e_7-\frac{22}{3}e_8$ and $a=\frac{2z-22}{3}$. Now we may restate Lemma 3.14.
\begin{replemma}{f}
Suppose $\lambda=-ae_6-ae_7+ae_8$ for some $a\in\mathbb{C}$. Write $\lambda=\lambda_0+z\zeta$ in the special line.
\begin{enumerate}[(1)]
\item If $z\notin1+\mathbb{Z}_{\geq0}$, then $M_\mathfrak{p}^\mathfrak{g}(\lambda)$ is simple.
\item If $z\in11+\mathbb{Z}_{\geq0}$, then $M_\mathfrak{p}^\mathfrak{g}(\lambda)$ is reducible.
\end{enumerate}
\end{replemma}
\begin{proof}
Because $a=\frac{2z-22}{3}$, the conclusion follows from Lemma 3.14 immediately.
\end{proof}
By Theorem 2.3, Lemma 12.3, and Theorem 12.4 in [\textbf{EHW}], one may check that $A(\lambda_0)=8$, $B(\lambda_0)=11$, and $C(\lambda_0)=3$ in this case.
\begin{proposition}
Suppose $\lambda=-ae_6-ae_7+ae_8$ for some $a\in\mathbb{C}$. Write $\lambda=\lambda_0+z\zeta$ in the special line. If $z=9$, then $M_\mathfrak{p}^\mathfrak{g}(\lambda)$ is reducible.
\end{proposition}
\begin{proof}
If $z=9$, then $a=-\frac{4}{3}$. For $\alpha_-(\mathrm{v})\in\Phi_\mathfrak{u}^+$, we have\[\langle\lambda+\rho,\alpha_-(\mathrm{v})\rangle=\frac{1}{2}((-1)^{v(2)}+2(-1)^{v(3)}+3(-1)^{v(4)}+4(-1)^{v(5)})+4.\] Then\[S_\lambda=\{\alpha_-(\mathrm{v})\in\Phi_\mathfrak{u}^+\mid\textrm{at least one of }v(3),v(4),\textrm{ and }v(5)\textrm{ is even}\}.\]We need to compute $s_\beta(\lambda+\rho)$. Here $\lambda+\rho=e_2+2e_3+3e_4+4e_5+(-a-4)e_6+(-a-4)e_7+(a+4)e_8$. We do not need to work out all the coefficients, but only the coefficients of $e_i$ for $1\leq i\leq5$. Moreover, we use five ``$+$'' and ``$-$'' symbols to indicate the parity of $v(i)$ for $1\leq i\leq5$ in $\beta$, where ``$+$'' corresponds to $v(i)$ even and ``$-$'' corresponds to $v(i)$ odd. For example, ``$++--+$'' represents $\frac{1}{2}(e_1+e_2-e_3-e_4+e_5-e_6-e_7+e_8)$, and ``$-----$'' represents $\frac{1}{2}(-e_1-e_2-e_3-e_4-e_5-e_6-e_7+e_8)$.

According to Table 1, it is immediate that $s_\beta(\lambda+\rho)$ is $\Phi_\mathfrak{l}$-regular integral precisely when $\beta$ is represented by ``$+++++$'', ``$--+++$'', ``$+--++$'', ``$+-+-+$'' or ``$+-++-$''. Because the elements in $W_\mathfrak{l}$ do not change the parity of the number of positive coefficients and do not change the number of zeros, it follows that there do not exist $\omega\in W_\mathfrak{l}$ and $\beta$ not represented by ``$+++++$'' such that $s_\beta(\lambda+\rho)=\omega s_{\beta_0}(\lambda+\rho)$ for $\beta_0$ represented by ``$+++++$''. Therefore, $\displaystyle{\sum_{\beta\in S_\lambda}Y(s_\beta(\lambda+\rho))}\neq0$ by Proposition 2.4, and hence the conclusion holds by Theorem 2.3.
\begin{table}
\begin{center}
\caption{Coefficients of $e_i$ for $1\leq i\leq5$ of $s_\beta(\lambda+\rho)$ for $\beta\in S_\lambda$ with $z=9$}
\begin{tabular}{|c|c|c|c|c|c|}
\hline
$\beta$&$e_1$&$e_2$&$e_3$&$e_4$&$e_5$\\
\hline
$+++++$&$-\frac{9}{2}$&$-\frac{7}{2}$&$-\frac{5}{2}$&$-\frac{3}{2}$&$-\frac{1}{2}$\\
\hline
$--+++$&$4$&$5$&$-2$&$-1$&$0$\\
\hline
$-+-++$&$\frac{7}{2}$&$-\frac{5}{2}$&$\frac{11}{2}$&$-\frac{1}{2}$&$\frac{1}{2}$\\
\hline
$-++-+$&$3$&$-2$&$-1$&$6$&$1$\\
\hline
$-+++-$&$\frac{5}{2}$&$-\frac{3}{2}$&$-\frac{1}{2}$&$\frac{1}{2}$&$\frac{13}{2}$\\
\hline
$+--++$&$-3$&$4$&$5$&$0$&$1$\\
\hline
$+-+-+$&$-\frac{5}{2}$&$\frac{7}{2}$&$-\frac{1}{2}$&$\frac{11}{2}$&$\frac{3}{2}$\\
\hline
$+-++-$&$-2$&$3$&$0$&$1$&$6$\\
\hline
$++--+$&$-2$&$-1$&$4$&$5$&$2$\\
\hline
$++-+-$&$-\frac{3}{2}$&$-\frac{1}{2}$&$\frac{7}{2}$&$\frac{3}{2}$&$\frac{11}{2}$\\
\hline
$+++--$&$-1$&$0$&$1$&$4$&$5$\\
\hline
$----+$&$\frac{3}{2}$&$\frac{5}{2}$&$\frac{7}{2}$&$\frac{9}{2}$&$\frac{5}{2}$\\
\hline
$---+-$&$1$&$2$&$3$&$2$&$5$\\
\hline
$--+--$&$\frac{1}{2}$&$\frac{3}{2}$&$\frac{3}{2}$&$\frac{7}{2}$&$\frac{9}{2}$\\
\hline
\end{tabular}
\end{center}
\end{table}
\end{proof}
\begin{proposition}
Suppose $\lambda=-ae_6-ae_7+ae_8$ for some $a\in\mathbb{C}$. Write $\lambda=\lambda_0+z\zeta$ in the special line. If $z=10$, then $M_\mathfrak{p}^\mathfrak{g}(\lambda)$ is reducible.
\end{proposition}
\begin{proof}
If $z=10$, then $a=-\frac{2}{3}$. For $\alpha_-(\mathrm{v})\in\Phi_\mathfrak{u}^+$, we have\[\langle\lambda+\rho,\alpha_-(\mathrm{v})\rangle=\frac{1}{2}((-1)^{v(2)}+2(-1)^{v(3)}+3(-1)^{v(4)}+4(-1)^{v(5)})+5.\] Then\[S_\lambda=\{\alpha_-(\mathrm{v})\in\Phi_\mathfrak{u}^+\mid\textrm{at least one of }v(2),v(3),v(4),\textrm{ and }v(5)\textrm{ is even}\}.\]Similar to Proposition 3.15, we need to compute the coefficients of $e_i$ for $1\leq i\leq5$ of $s_\beta(\lambda+\rho)$.

According to Table 2, it is immediate that $s_\beta(\lambda+\rho)$ is $\Phi_\mathfrak{l}$-regular precisely when $\beta$ is represented by ``$+++++$'', ``$--+++$'', ``$-+-++$'', ``$-++-+$'' or ``$-+++-$''. For the same reason as in Proposition 3.15, $Y(s_{\beta_0}(\lambda+\rho))$ for $\beta_0$ represented by ``$-++-+$'' leads to $\displaystyle{\sum_{\beta\in S_\lambda}Y(s_\beta(\lambda+\rho))}\neq0$ by Proposition 2.4, and hence the conclusion holds by Theorem 2.3.
\begin{table}
\begin{center}
\caption{Coefficients of $e_i$ for $1\leq i\leq5$ of $s_\beta(\lambda+\rho)$ for $\beta\in S_\lambda$ with $z=10$}
\begin{tabular}{|c|c|c|c|c|c|}
\hline
$\beta$&$e_1$&$e_2$&$e_3$&$e_4$&$e_5$\\
\hline
$+++++$&$-5$&$-4$&$-3$&$-2$&$-1$\\
\hline
$--+++$&$\frac{9}{2}$&$\frac{11}{2}$&$-\frac{5}{2}$&$-\frac{3}{2}$&$-\frac{1}{2}$\\
\hline
$-+-++$&$4$&$-3$&$6$&$-1$&$0$\\
\hline
$-++-+$&$\frac{7}{2}$&$-\frac{5}{2}$&$-\frac{3}{2}$&$\frac{13}{2}$&$\frac{1}{2}$\\
\hline
$-+++-$&$3$&$-2$&$-1$&$0$&$7$\\
\hline
$+--++$&$-\frac{7}{2}$&$\frac{9}{2}$&$\frac{11}{2}$&$-\frac{1}{2}$&$\frac{1}{2}$\\
\hline
$+-+-+$&$-3$&$4$&$-1$&$6$&$1$\\
\hline
$+-++-$&$-\frac{5}{2}$&$\frac{7}{2}$&$-\frac{1}{2}$&$\frac{1}{2}$&$\frac{13}{2}$\\
\hline
$++--+$&$-\frac{5}{2}$&$-\frac{3}{2}$&$\frac{9}{2}$&$\frac{11}{2}$&$\frac{3}{2}$\\
\hline
$++-+-$&$-2$&$-1$&$4$&$1$&$6$\\
\hline
$+++--$&$-\frac{3}{2}$&$-\frac{1}{2}$&$\frac{1}{2}$&$\frac{9}{2}$&$\frac{11}{2}$\\
\hline
$----+$&$2$&$3$&$4$&$5$&$2$\\
\hline
$---+-$&$\frac{3}{2}$&$\frac{5}{2}$&$\frac{7}{2}$&$\frac{3}{2}$&$\frac{11}{2}$\\
\hline
$--+--$&$1$&$2$&$1$&$4$&$5$\\
\hline
$-+---$&$\frac{1}{2}$&$\frac{1}{2}$&$\frac{5}{2}$&$\frac{7}{2}$&$\frac{9}{2}$\\
\hline
\end{tabular}
\end{center}
\end{table}
\end{proof}
\begin{theorem}
Suppose $\lambda=-ae_6-ae_7+ae_8$ for some $a\in\mathbb{C}$. Then $M_\mathfrak{p}^\mathfrak{g}(\lambda)$ is reducible if and only if $a\in-2+\frac{2}{3}\mathbb{Z}_{\geq0}$.
\end{theorem}
\begin{proof}
Write $\lambda=\lambda_0+z\zeta$ in the special line, and then Lemma 3.14$'$(1) and Theorem 2.6(1) show that $M_\mathfrak{p}^\mathfrak{g}(\lambda)$ is reducible only if $z\in8+\mathbb{Z}_{\geq0}$. On the other hand, Lemma 3.14$'$(2), Proposition 3.15, Proposition 3.16, and Theorem 2.6(2) show that if $z\in8+\mathbb{Z}_{\geq0}$, then $M_\mathfrak{p}^\mathfrak{g}(\lambda)$ is reducible. It follows that $M_\mathfrak{p}^\mathfrak{g}(\lambda)$ is reducible if and only if $z\in8+\mathbb{Z}_{\geq0}$, which is equivalent to $a\in-2+\frac{2}{3}\mathbb{Z}_{\geq0}$.
\end{proof}

\subsection{$(E_{7(-133)},SO(2)\times E_{6(-78)})$}

Let $\mathfrak{g}=\mathfrak{e}_7$ and let $\mathfrak{p}=\mathfrak{l}+\mathfrak{u}_+$ be a parabolic subalgebra of abelian type with $\mathfrak{l}=\mathfrak{so}(2,\mathbb{C})\oplus\mathfrak{e}_6$. We may choose a Cartan subalgebra $\mathfrak{h}\subseteq \mathfrak{l}$ and a simple root system $\Delta=\{\alpha_i\mid1\leq i\leq7\}$ given by the Dynkin diagram of Figure 2 such that $\mathfrak{p}$ is standard with respect to $(\mathfrak{h},\Delta)$. Then $\Delta_\mathfrak{l}=\{\alpha_i\mid1\leq i\leq6\}$ of $\Delta$.
\begin{figure}[H]
\centering \scalebox{0.7}{\includegraphics{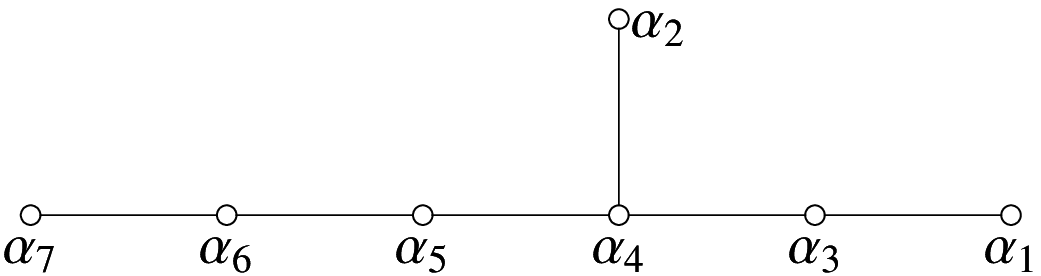}}
\caption{Dynkin diagram of $\mathfrak{e}_7$.}
\end{figure}
Embed $\mathfrak{h}_\mathbb{R}^*$, the $\mathbb{R}$-span of the simple roots, into the subspace $V_7:=\{v\in\mathbb{R}^8\mid(v,e_7+e_8)=0\}$ of $\mathbb{R}^8$. Then $\Phi_\mathfrak{l}^+$ equals the full set of the positive roots of $\mathfrak{e}_6$ as in Section 3.6, while $\alpha_7=e_6-e_5$ and \[\Phi_\mathfrak{u}^+=\{e_6\pm e_i\mid1\leq i\leq5\}\cup\{e_8-e_7\}\cup\{\alpha_+(\mathrm{v})\mid\displaystyle{\sum_{i=1}^5}v(i)\textrm{ odd}\}.\]Here we retain the notations $\alpha_\pm(\mathrm{v})$ as in \eqref{i}. Moreover, we have $\rho=e_2+2e_3+3e_4+4e_5+5e_6-\frac{17}{2}e_7+\frac{17}{2}e_8$.

If $M_\mathfrak{p}^\mathfrak{g}(\lambda)$ is of scalar type, an easy computation shows that $\lambda=ae_6-\frac{1}{2}ae_7+\frac{1}{2}ae_8$ for some $a\in\mathbb{C}$.
\begin{lemma}
\label{g}
Suppose $\lambda=ae_6-\frac{1}{2}ae_7+\frac{1}{2}ae_8$ for some $a\in\mathbb{C}$.
\begin{enumerate}[(1)]
\item If $a\notin-16+\mathbb{Z}_{\geq0}$, then $M_\mathfrak{p}^\mathfrak{g}(\lambda)$ is simple.
\item If $a\in\mathbb{Z}_{\geq0}$, then $M_\mathfrak{p}^\mathfrak{g}(\lambda)$ is reducible.
\end{enumerate}
\end{lemma}
\begin{proof}
It is obvious that $\lambda+\rho$ is always $\Phi_\mathfrak{l}^+$-dominant integral. Firstly, $\langle\lambda+\rho,e_8-e_7\rangle=a+17$. Secondly, $\langle\lambda+\rho,e_6\pm e_i\rangle=a+5\pm(i-1)$ for $1\leq i\leq5$. Thirdly, for $\alpha_+(\mathrm{v})\in\Phi_\mathfrak{u}^+$, one computes that $\langle\lambda+\rho,\alpha_+(\mathrm{v})\rangle=a+11+\frac{1}{2}((-1)^{v(2)}+2(-1)^{v(3)}+3(-1)^{v(4)}+4(-1)^{v(5)})$. Now the conclusion follows from Theorem 2.1 immediately.
\end{proof}
The maximal root $\gamma$ in $\Phi^+$ is $e_8-e_7$, and it follows that $\zeta=e_6-\frac{1}{2}e_7+\frac{1}{2}e_8$. If we write $\lambda=\lambda_0+z\zeta$ in the special line, then we obtain $\lambda_0=-17e_6+\frac{17}{2}e_7-\frac{17}{2}e_8$ and $a=z-17$. Now we may restate Lemma 3.18.
\begin{replemma}{g}
Suppose $\lambda=ae_6-\frac{1}{2}ae_7+\frac{1}{2}ae_8$ for some $a\in\mathbb{C}$. Write $\lambda=\lambda_0+z\zeta$ in the special line.
\begin{enumerate}[(1)]
\item If $z\notin1+\mathbb{Z}_{\geq0}$, then $M_\mathfrak{p}^\mathfrak{g}(\lambda)$ is simple.
\item If $z\in17+\mathbb{Z}_{\geq0}$, then $M_\mathfrak{p}^\mathfrak{g}(\lambda)$ is reducible.
\end{enumerate}
\end{replemma}
\begin{proof}
Because $a=z-17$, the conclusion follows from Lemma 3.18 immediately.
\end{proof}
By Theorem 2.3, Lemma 13.3, and Theorem 13.4 in [\textbf{EHW}], one may check that $A(\lambda_0)=9$, $B(\lambda_0)=17$, and $C(\lambda_0)=4$ in this case.

Actually, we only need to check $z\in\{10,11,12,14,15,16\}$. However, the Weyl group of $W_\mathfrak{l}$ is too complicated. It is hard to verify whether $\displaystyle{\sum_{\beta\in S_\lambda}Y(s_\beta(\lambda+\rho))}$ equals 0 by means of Proposition 2.4. Therefore, we provide an alternate method to solve the problem for $\mathfrak{e}_7$.

\begin{lemma}
Let $\alpha_u$ be the unique simple root in $\Phi_\mathfrak{u}^+$, and denote by $\theta_u$ the fundamental weight of $\alpha_u$. Suppose $\mu,\nu\in\mathfrak{h}^*$. If $\mu=\omega\nu$ for some $\omega\in W_\mathfrak{l}$, then $(\mu,\theta_u)=(\nu,\theta_u)$.
\end{lemma}
\begin{proof}
Because $\omega\in W_\mathfrak{l}$, $\omega\alpha\in\Phi_\mathfrak{l}$ for all $\alpha\in\Delta_\mathfrak{l}$. Hence $\langle\omega\theta_u,\alpha\rangle=\frac{2(\omega\theta_u,\alpha)}{(\alpha,\alpha)}=\frac{2(\theta_u,\omega^{-1}\alpha)}{(\alpha,\alpha)}=0$ for all $\alpha\in\Delta_\mathfrak{l}$. On the other hand, $\omega^{-1}\alpha_u=\alpha_u+\beta$ for some $\beta\in\mathrm{span}_\mathbb{Z}\Delta_\mathfrak{l}$, and it follows that\[\langle\omega\theta_u,\alpha_u\rangle=\frac{2(\omega\theta_u,\alpha_u)}{(\alpha_u,\alpha_u)}=\frac{2(\theta_u,\omega^{-1}\alpha_u)}{(\alpha_u,\alpha_u)}= \frac{2(\theta_u,\alpha_u+\beta)}{(\alpha_u,\alpha_u)}=\frac{2(\theta_u,\alpha_u)}{(\alpha_u,\alpha_u)}=\langle\theta_u,\alpha_u\rangle=1.\]Therefore, $\omega\theta_u=\theta_u$. Now if $\mu=\omega\nu$ for some $\omega\in W_\mathfrak{l}$, then $(\mu,\theta_u)=(\omega\nu,\theta_u)=(\nu,\omega^{-1}\theta_u)=(\nu,\theta_u)$.
\end{proof}
\begin{proposition}
Suppose $\lambda=ae_6-\frac{1}{2}ae_7+\frac{1}{2}ae_8$ for some $a\in\mathbb{C}$. Write $\lambda=\lambda_0+z\zeta$ in the special line. If $z\in\{12,14,15,16\}$, then $M_\mathfrak{p}^\mathfrak{g}(\lambda)$ is reducible.
\end{proposition}
\begin{proof}
The only simple root in $\Phi_\mathfrak{u}^+$ is $\alpha_7=e_6-e_5$, and the fundamental weight $\theta_7$ of $\alpha_7$ is $e_6-\frac{1}{2}e_7+\frac{1}{2}e_8$. It is easy to write down that $\lambda+\rho=e_2+2e_3+3e_4+4e_5+(a+5)e_6-\frac{a+17}{2}e_7+\frac{a+17}{2}e_8$. In order to apply Proposition 2.4 via the contrapositive of Lemma 3.19, we need to compute $(s_\beta(\lambda+\rho),\theta_7)$ for $\beta\in\Phi_\mathfrak{u}^+$. It is almost immediate to work out that\[(s_{e_6\pm e_i}(\lambda+\rho),\theta_7)=\frac{a+17}{2}\pm(i-1)\]for $1\leq i\leq5$ and\[(s_{e_8-e_7}(\lambda+\rho),\theta_7)=\frac{a-7}{2}.\]As for the roots in $\{\alpha_+(\mathrm{v})\mid\displaystyle{\sum_{i=1}^5v(i)}\textrm{ odd}\}$, let us still use the notation as in Proposition 3.15. Because $\theta_7$ only involves $e_6$, $e_7$, and $e_8$, we just write down the coefficients of these three vectors for each $s_\beta(\lambda+\rho)$. The values are listed in Table 3.
\begin{table}
\begin{center}
\caption{Coefficients of $e_6$, $e_7$, and $e_8$ of $s_\beta(\lambda+\rho)$ and $(s_\beta(\lambda+\rho),\theta_7)$}
\begin{tabular}{|c|c|c|c|c|c|}
\hline
$\beta$&$e_6$&$e_7$&$e_8$&$(s_\beta(\lambda+\rho),\theta_7)$\\
\hline
$-++++$&$\frac{a}{2}-3$&$-\frac{1}{2}$&$\frac{1}{2}$&$\frac{a-5}{2}$\\
\hline
$+-+++$&$\frac{a-5}{2}$&$-1$&$1$&$\frac{a-3}{2}$\\
\hline
$++-++$&$\frac{a}{2}-2$&$-\frac{3}{2}$&$\frac{3}{2}$&$\frac{a-1}{2}$\\
\hline
$+++-+$&$\frac{a-3}{2}$&$-2$&$2$&$\frac{a+1}{2}$\\
\hline
$++++-$&$\frac{a}{2}-1$&$-\frac{5}{2}$&$\frac{5}{2}$&$\frac{a+3}{2}$\\
\hline
$---++$&$\frac{a-3}{2}$&$-2$&$2$&$\frac{a+1}{2}$\\
\hline
$--+-+$&$\frac{a}{2}-1$&$-\frac{5}{2}$&$\frac{5}{2}$&$\frac{a+3}{2}$\\
\hline
$--++-$&$\frac{a-1}{2}$&$-3$&$3$&$\frac{5}{2}$\\
\hline
$-+--+$&$\frac{a-1}{2}$&$-3$&$3$&$\frac{5}{2}$\\
\hline
$-+-+-$&$\frac{a}{2}$&$-\frac{7}{2}$&$\frac{7}{2}$&$\frac{a+7}{2}$\\
\hline
$-++--$&$\frac{a+1}{2}$&$-4$&$4$&$\frac{a+9}{2}$\\
\hline
$+---+$&$\frac{a}{2}$&$-\frac{7}{2}$&$\frac{7}{2}$&$\frac{a+7}{2}$\\
\hline
$+--+-$&$\frac{a+1}{2}$&$-4$&$4$&$\frac{a+9}{2}$\\
\hline
$+-+--$&$\frac{a}{2}+1$&$-\frac{9}{2}$&$\frac{9}{2}$&$\frac{a+11}{2}$\\
\hline
$++---$&$\frac{a+3}{2}$&$-5$&$5$&$\frac{a+13}{2}$\\
\hline
$-----$&$\frac{a}{2}+2$&$-\frac{11}{2}$&$\frac{11}{2}$&$\frac{a+15}{2}$\\
\hline
\end{tabular}
\end{center}
\end{table}
According to Table 3, it is immediate that\[(s_{e_8-e_7}(\lambda+\rho),\theta_7)=\frac{a-7}{2}\neq(s_\beta(\lambda+\rho),\theta_7)\]for all $\beta\in\{\alpha_+(\mathrm{v})\mid\displaystyle{\sum_{i=1}^5}v(i)\textrm{ odd}\}$. Moreover,\[(s_{e_8-e_7}(\lambda+\rho),\theta_7)=\frac{a-7}{2}\neq\frac{a+17}{2}\pm(i-1)\]for $1\leq i\leq5$. Hence by Lemma 3.19, there does not exist $\omega\in W_\mathfrak{l}$ such that $\omega s_{e_8-e_7}(\lambda+\rho)$ equals $s_\beta(\lambda+\rho)$ for any other $\beta\in\Phi_\mathfrak{u}^+$. On the other hand, if $z\in\{12,14,15,16\}$, then $\langle\lambda+\rho,e_8-e_7\rangle=a+17=z\in\mathbb{Z}_{>0}$, so $e_8-e_7\in S_\lambda$.

What remains to prove is that $Y(s_{e_8-e_7}(\lambda+\rho))\neq0$, and $Y(s_{e_8-e_7}(\lambda+\rho))$ leads to $\displaystyle{\sum_{\beta\in S_\lambda}Y(s_\beta(\lambda+\rho))}\neq0$ by Proposition 2.4, so the conclusion follows from Theorem 2.3. In fact, it is obvious that $(s_{e_8-e_7}(\lambda+\rho),e_j\pm e_i)=j-i\in\mathbb{Z}\setminus\{0\}$ for $1\leq i<j\leq5$. For $\alpha=\alpha_-(\mathrm{v})\in\Phi_\mathfrak{l}^+$, we have\[(s_{e_8-e_7}(\lambda+\rho),\alpha)=\frac{1}{2}(-1)^{v(2)}+(-1)^{v(3)}+\frac{3}{2}(-1)^{v(4)}+2(-1)^{v(5)}-11-a.\]If $z\in\{12,14,15,16\}$, then $a\in\{-5,-3,-2,-1\}$ because $a=z-17$. Thus for all choices of values of $v(i)$ for $2\leq i\leq5$, $\frac{1}{2}(-1)^{v(2)}+(-1)^{v(3)}+\frac{3}{2}(-1)^{v(4)}+2(-1)^{v(5)}-11-a\in\mathbb{Z}\setminus\{0\}$. This shows that $s_{e_8-e_7}(\lambda+\rho)$ is $\Phi_\mathfrak{l}$-regular integral, and $Y(s_{e_8-e_7}(\lambda+\rho))\neq0$ by Proposition 2.2(1).
\end{proof}
\begin{proposition}
Suppose $\lambda=ae_6-\frac{1}{2}ae_7+\frac{1}{2}ae_8$ for some $a\in\mathbb{C}$. Write $\lambda=\lambda_0+z\zeta$ in the special line. If $z\in\{10,11\}$, then $M_\mathfrak{p}^\mathfrak{g}(\lambda)$ is reducible.
\end{proposition}
\begin{proof}
Recall that $\lambda+\rho=e_2+2e_3+3e_4+4e_5+(a+5)e_6-\frac{a+17}{2}e_7+\frac{a+17}{2}e_8$. Assume $z=10$ first, and then $a=z-17=-7$. Recall in the proof of Proposition 3.20 that $(s_{e_8-e_7}(\lambda+\rho),\alpha)=\frac{1}{2}(-1)^{v(2)}+(-1)^{v(3)}+\frac{3}{2}(-1)^{v(4)}+2(-1)^{v(5)}-4$ for $\alpha=\alpha_-(\mathrm{v})\in\Phi_\mathfrak{l}^+$, and if $v(1)=v(2)=1$ and $v(3)=v(4)=v(5)=0$, then $(s_{e_8-e_7}(\lambda+\rho),\alpha)=0$. Thus $s_{e_8-e_7}(\lambda+\rho)$ is $\Phi_\mathfrak{l}$-singular. Next consider $s_{e_6\pm e_i}(\lambda+\rho)$ for $1\leq i\leq5$. If $i\neq3$, then $(s_{e_6\pm e_i}(\lambda+\rho),e_2\pm e_i)=0$. Hence $s_{e_6\pm e_i}(\lambda+\rho)$ is $\Phi_\mathfrak{l}$-regular only if $i=3$. But $\langle\lambda+\rho,e_6\pm e_3\rangle\notin\mathbb{Z}^+$, so $e_6\pm e_3\notin S_\lambda$. Therefore, we only need to consider the roots in $\Phi_\mathfrak{u}^+$ of the form $\alpha_+(\mathrm{v})$ with $\displaystyle{\sum_{i=1}^5v(i)}$ odd.
\begin{table}
\begin{center}
\caption{Coefficients of $e_i$ for $1\leq i\leq5$ of $s_\beta(\lambda+\rho)$}
\begin{tabular}{|c|c|c|c|c|c|}
\hline
$\beta$&$e_1$&$e_2$&$e_3$&$e_4$&$e_5$\\
\hline
$-++++$&$\frac{a}{2}+8$&$-\frac{a}{2}-7$&$-\frac{a}{2}-6$&$-\frac{a}{2}-5$&$-\frac{a}{2}-4$\\
\hline
$+-+++$&$-\frac{a+15}{2}$&$\frac{a+17}{2}$&$-\frac{a+11}{2}$&$-\frac{a+9}{2}$&$-\frac{a+7}{2}$\\
\hline
$++-++$&$-\frac{a}{2}-7$&$-\frac{a}{2}-6$&$\frac{a}{2}+9$&$-\frac{a}{2}-4$&$-\frac{a}{2}-3$\\
\hline
$+++-+$&$-\frac{a+13}{2}$&$-\frac{a+11}{2}$&$-\frac{a+9}{2}$&$\frac{a+17}{2}$&$-\frac{a+5}{2}$\\
\hline
$++++-$&$-\frac{a}{2}-6$&$-\frac{a}{2}-5$&$-\frac{a}{2}-4$&$-\frac{a}{2}-3$&$\frac{a}{2}+10$\\
\hline
$---++$&$\frac{a+13}{2}$&$\frac{a+15}{2}$&$\frac{a+17}{2}$&$-\frac{a+7}{2}$&$-\frac{a+5}{2}$\\
\hline
$--+-+$&$\frac{a}{2}+6$&$\frac{a}{2}+7$&$-\frac{a}{2}-4$&$\frac{a}{2}+9$&$-\frac{a}{2}-2$\\
\hline
$--++-$&$\frac{a+11}{2}$&$\frac{a+13}{2}$&$-\frac{a+7}{2}$&$-\frac{a+5}{2}$&$\frac{a+19}{2}$\\
\hline
$-+--+$&$\frac{a+11}{2}$&$-\frac{a+9}{2}$&$\frac{a+15}{2}$&$\frac{a+17}{2}$&$-\frac{a+3}{2}$\\
\hline
$-+-+-$&$\frac{a}{2}+5$&$-\frac{a}{2}-4$&$\frac{a}{2}+7$&$-\frac{a}{2}-2$&$\frac{a}{2}+9$\\
\hline
$-++--$&$\frac{a+9}{2}$&$-\frac{a+7}{2}$&$-\frac{a+5}{2}$&$\frac{a+15}{2}$&$\frac{a+17}{2}$\\
\hline
$+---+$&$-\frac{a}{2}-5$&$\frac{a}{2}+6$&$\frac{a}{2}+7$&$\frac{a}{2}+8$&$-\frac{a}{2}-1$\\
\hline
$+--+-$&$-\frac{a+9}{2}$&$\frac{a+11}{2}$&$\frac{a+13}{2}$&$-\frac{a+3}{2}$&$\frac{a+17}{2}$\\
\hline
$+-+--$&$-\frac{a}{2}-4$&$\frac{a}{2}+5$&$-\frac{a}{2}-2$&$\frac{a}{2}+7$&$\frac{a}{2}+8$\\
\hline
$++---$&$-\frac{a+7}{2}$&$-\frac{a+5}{2}$&$\frac{a+11}{2}$&$\frac{a+13}{2}$&$\frac{a+15}{2}$\\
\hline
\end{tabular}
\end{center}
\end{table}
Table 4 lists the coefficients of $e_i$ for $1\leq i\leq5$ of $s_\beta(\lambda+\rho)$ for $\beta\in\{\alpha_+(\mathrm{v})\mid\displaystyle{\sum_{i=1}^5v(i)}\textrm{ odd}\}$. We exclude the case ``$-----$'' in Table 4 because the root represented by it does not lie in $S_\lambda$ for both $z=10$ and $z=11$. Moreover, for $z=10$, i.e., $a=-7$, the root represented by ``$++---$'' is also excluded for the same reason. Now according to Table 4, one checks immediately that there does not exist $e_j\pm e_i$ for $1\leq i<j\leq5$ such that $(s_\beta(\lambda+\rho),e_j\pm e_i)=0$ only if $\beta$ is represented by one of the following five sign patterns
\begin{equation}
\label{h}
-++++, ~~ +-+++, ~~ ---++, ~~ --+-+, ~~ --++-.
\end{equation}
Therefore, we only need to consider these five roots, which are the only possible roots $\beta$ such that $s_\beta(\lambda+\rho)$ are $\Phi_\mathfrak{l}$-regular in $S_\lambda$.

Let us consider the root represented by ``$-++++$'', i.e.,\[\beta_0=\frac{1}{2}(-e_1+e_2+e_3+e_4+e_5+e_6-e_7+e_8).\]First, $\langle\lambda+\rho,\beta_0\rangle=9\in\mathbb{Z}_{>0}$ shows that $\beta_0\in S_\lambda$. Second, according to Table 4, we know that\[s_{\beta_0}(\lambda+\rho)=\frac{1}{2}(9e_1-7e_2-5e_3-3_4-e_5-13e_6-e_7+e_8).\]It is obvious that $(s_{\beta_0}(\lambda+\rho),e_j\pm e_i)\in\mathbb{Z}\setminus\{0\}$ for $1\leq i<j\leq5$. For a root of the form $\alpha_-(\mathrm{v})$ in $\Phi_\mathfrak{l}^+$,\[(\alpha_-(\mathrm{v}),s_{\beta_0}(\lambda+\rho))=\frac{1}{4}(9(-1)^{v(1)}-7(-1)^{v(2)}-5(-1)^{v(3)}-3(-1)^{v(4)}-(-1)^{v(5)}+15)\]which equals 0 if and only if $v(1)$ and $v(3)$ are even, while $v(2)$, $v(4)$, and $v(5)$ are odd. But this is not a root in $\Phi_\mathfrak{l}$. This shows that $s_{\beta_0}(\lambda+\rho)$ is $\Phi_\mathfrak{l}$-regular integral. Hence $Y(s_{\beta_0}(\lambda+\rho))\neq0$ by Proposition 2.2(1).

According to Table 3, we know that if $\beta_1\neq\beta_2$ where $\beta_1$ and $\beta_2$ are chosen from the five roots in \eqref{h}, then $(\beta_1,\theta_7)\neq(\beta_2,\theta_7)$. Hence by Lemma 3.19, there does not exist $\omega\in W_\mathfrak{l}$ such that $s_{\beta_0}(\lambda+\rho)=\omega s_{\beta}(\lambda+\rho)$ for any other root $\beta$ chosen from the five roots in \eqref{h}. By Proposition 2.4 and Theorem 2.3, the conclusion holds in the case of $z=10$.

The proof for $z=11$ is parallel. One shows that the only possible roots $\beta$ such that $s_\beta(\lambda+\rho)$ are $\Phi_\mathfrak{l}$-regular in $S_\lambda$ are just those represented by ``$-++++$'', ``$+-+++$'', ``$++-++$'', ``$+++-+$'', and ``$++++-$'', and the root represented by ``$+-+++$'' can be chosen as $\beta_0$ as in the case of $z=10$.
\end{proof}
\begin{theorem}
Suppose $\lambda=ae_6-\frac{1}{2}ae_7+\frac{1}{2}ae_8$ for some $a\in\mathbb{C}$. Then $M_\mathfrak{p}^\mathfrak{g}(\lambda)$ is reducible if and only if $a\in-8+\mathbb{Z}_{\geq0}$.
\end{theorem}
\begin{proof}
Write $\lambda=\lambda_0+z\zeta$ in the special line, and then Lemma 3.18$'$(1) and Theorem 2.6(1) show that $M_\mathfrak{p}^\mathfrak{g}(\lambda)$ is reducible only if $z\in9+\mathbb{Z}_{\geq0}$. On the other hand, Lemma 3.18$'$(2), Proposition 3.20, Proposition 3.21, and Theorem 2.6(2) show that if $z\in9+\mathbb{Z}_{\geq0}$, then $M_\mathfrak{p}^\mathfrak{g}(\lambda)$ is reducible. It follows that $M_\mathfrak{p}^\mathfrak{g}(\lambda)$ is reducible if and only if $z\in9+\mathbb{Z}_{\geq0}$, which is equivalent to $a\in-8+\mathbb{Z}_{\geq0}$.
\end{proof}

\begin{center}
\textbf{Acknowledgements}
\end{center}
I felt interested in the reducibility of generalized Verma modules when discussing with Doctor Zhanqiang BAI on another research problem which inspired me with the ideas. In the process of the study, Doctor Toshihisa KUBO sent me his Ph.D. thesis which offered me some useful techniques. After I finished the paper, some reviewers gave me very concrete and thoughtful advice on revising the paper. I would like to express my thankfulness to all of them.


\begin{thebibliography} {EHW}
\bibitem[B]{B} B. D. Boe, \textit{Homomorphisms between Generalized Verma Modules}, Transactions of the American Mathematical Society, Volume 288 (1985), Number 2, Page 791--799.
\bibitem[EHW]{EHW} T. Enright, R. Howe, and N. Wallach, \textit{A Classification of Unitary Highest Weight Modules}, Representation Theory of Reductive Groups, Progress in Mathematics, Volume 40 (1983), Page 97--143, Birkh\"{a}user, Boston-Basel-Stuttgart.
\bibitem[F]{F} P. Franek, \textit{Generalized Verma Module Homomorphisms in Singular Character}, Archivum Mathematicum, Tomus 42 (2006), Supplement, Page 229--240.
\bibitem[G]{G} A. Gyoja, \textit{A Remark on Homomorphisms between Generalized Verma Modules}, Journal of Mathematics of Kyoto University, Volume 34 (1994), Number 4, Page 695--697.
\bibitem[HC1]{HC1} Harish-Chandra, \textit{Representations of Semisimple Lie groups IV}, American Journal of Mathematics, Volume 77 (1955), Page 743--777.
\bibitem[HC2]{HC2} Harish-Chandra, \textit{Representations of Semisimple Lie groups V}, American Journal of Mathematics, Volume 78 (1956), Page 1--41.
\bibitem[H]{H} J. E. Humphreys, \textit{Representations of Semisimple Lie Algebras in the BGG Category $\mathcal{O}$}, Graduate Studies in Mathematics, Volume 94 (2008), American Mathematical Society, Providence, Rhode Island.
\bibitem[J]{J} J. C. Jantzen, \textit{Kontravariante Formen auf induzierten Darstellungen halbeinfacher Lie-Algebren}, Mathematische Annalen, Volume 226 (1977), Page 53--65.
\bibitem[Ka]{Ka} A. Kamita, \textit{The $b$-functions for Prehomogeneous Vector Spaces of Commutative Parabolic Type and Universal Generalized Verma Modules}, Publications of the Research Institute for Mathematical Sciences, Kyoto University, Volume 41 (2005), Page 471--495.
\bibitem[Kn]{Kn} A. W. Knapp, \textit{Lie Groups Beyond an Introduction (Second Edition)}, Progress in Mathematics, Volume 140 (2002), Birkh\"{a}user, Boston-Basel-Berlin.
\bibitem[Ku]{Ku} T. Kubo, \textit{Conformally Invariant Systems of Differential Operators Associative to Two-Step Nilpotent Maximal Parabolics of Non-Heisenberg Type}, Ph.D. Thesis Submitted to Oklahoma State University for the Degree of Doctor of Philosophy (2012).
\bibitem[M]{M} H. Matumoto, \textit{The homomorphisms between scalar generalized Verma modules associated to maximal parabolic subalgebras}, Duke Mathematical Journal, Volume 131 (2006), Number 1, Page 75--118.
\bibitem[RRS]{RRS} R. Richardson, G. R\"{o}hrle, R. Steinberg, \textit{Parabolic subgroup with abelian unipotent radical}, Inventiones Mathematicae, Volume 110 (1992), Page 649--671.
\bibitem[W]{W} J. A. Wolf, \textit{On the Classification of Hermitian Symmetric Spaces}, Indiana University Mathematics Journal, Volume 13 (1964), Page 489--495.
\end{thebibliography}
\end{document}